\documentclass{amsart}
\usepackage{mathrsfs}
\usepackage{float}
\usepackage{url}

\newcommand{\C}{\mathbb{C}}
\newcommand{\K}{\mathbb{K}}
\newcommand{\LL}{\mathbb{L}}
\newcommand{\PP}{\mathbb{P}}
\newcommand{\Q}{\mathbb{Q}}
\newcommand{\Z}{\mathbb{Z}}
\newcommand{\y}{{\scriptstyle\mathscr{Y}}}
\newcommand{\ua}{\underline{a}}
\newcommand{\res}{\mathrm{res}}

\newtheorem{proposition}{Proposition}[section]
\newtheorem*{proposition*}{Proposition}
\newtheorem{lemma}[proposition]{Lemma}
\newtheorem{theorem}[proposition]{Theorem}
\newtheorem{corollary}[proposition]{Corollary}
\theoremstyle{definition}
\newtheorem{example}[proposition]{Example}
\newtheorem{definition}[proposition]{Definition}
\newtheorem{remark}[proposition]{Remark}

\makeatletter
\def\l@subsection{\@tocline{2}{0pt}{2.5pc}{5pc}{}}
\makeatother

\numberwithin{equation}{section}
\numberwithin{figure}{section}

\begin{document}

\title{Notes on Three Formulas of Abel}

\author{David A. Cox}

\address{Department of Mathematics \& Statistics, Amherst College, Amherst, MA 01002 USA}
\email{dacox@amherst.edu}


\begin{abstract}
These notes explore three amazing formulas proved by Abel in his 1826 Paris memoir on what we now call Abelian integrals.  We discuss the first two formulas from the point of view of symbolic computation and explain their connection to residues and partial fractions.  The third formula arises from the first two and is related to the genus and lattice points in the Newton polygon.
\end{abstract}

\thanks{I  thank the late Harold Edwards for introducing me to Abel's Paris memoir.  I am grateful to Alicia Dickenstein for her encouragement, and I thank the reviewer for many useful suggestions.}

\maketitle

\tableofcontents

\section{Abel's Paris Memoir}
\label{SecAbelMem}

In October of 1826, Niels Henrik Abel submitted \emph{M\'emoire sur une propri\'et\'e g\'en\'erale d'une classe tr\`es \'etendue de fonctions transcendantes} \cite{A1} to the Paris academy.\footnote{See \cite{DelC} for a facsimile reproduction of Abel's original manuscript.}  This amazing paper contains a number of results that fall under the rubric ``Abel's Theorem.''   The purpose of these notes is to explore one of the more underappreciated versions of Abel's Theorem and its connection to the genus of the associated Riemann surface. 

Abel starts with an equation written the form
\[
0 = p_0 + p_1 y + p_2 y^2 + \cdots + p_{n-1}y^{n-1} + y^n = \chi y.
\]
This is equation (1) on p.\ 146\footnote{Page numbers for Abel's papers refer to the versions in his collected works \cite[Vol.\ 1]{A5}.} of \cite{A1}.  The $p_i = p_i(x)$ are polynomials in $x$ (``fonctions rationnelles et enti\`eres d'une m\^{e}me quantit\'e variable $x$'').  Abel does not specify a field.  We will work over a computable\footnote{\emph{Computable} means that elements of $K$ have explicit representations where field operations can be done algorithmically.  See the discussion of \emph{effective field} in \cite[pp.\ 5-6]{Mora}} subfield $\K \subseteq \C$, so that $\chi y = \chi(x,y) \in \K[x,y]$.   We assume that $\chi(x,y)$ is irreducible over $\C[x,y]$. 

Given a rational function $f(x,y)$ with coefficients in $\K$, Abel's main object of interest is the indefinite integral
\begin{equation}
\label{AbInt}
\int\! f(x,y)\, dx,
\end{equation}
where $y$ is satisfies $\chi(x,y) = 0$.  These are called \emph{Abelian integrals} because of Abel's memoir.  In 1832, Jacobi used the term ``transcendentes Abelianis''  \cite{Jacobi},\footnote{Abel's memoir was not available when Jacobi wrote \cite{Jacobi} in 1832, but he was aware of its existence because of a footnote in the paper \cite{A3} Abel published in 1828.  Jacobi wrote Legendre in 1829 \cite{Jacobi4} to inquire about the status of Abel's memoir, and in a footnote on p.\ 397 of \cite{Jacobi}, he urged the Paris Academy to publish the memoir as a way to honor Abel, who died in 1829.} and in 1847, he wrote ``\emph{Abel}schen integrale"  \cite{Jacobi3}. In 1857, Riemann used ``Abel’schen Functionen"  \cite{Riemann}.  Abel certainly thought of \eqref{AbInt} as a function.  We will say more about this in Section \ref{Limits}.

In equation (2) on p.\ 147 of \cite{A1}, Abel introduces the auxiliary polynomial
\[
\theta y = q_0 + q_1 y + q_2 y^2 + \cdots + q_{n-1}y^{n-1}.
\]
Among the polynomials $q_i = q_i(x)$, a ``certain nombre" of the coefficients of $x$ are assumed to be indeterminates, which we write as $\ua = a_1, a_2, a_3, \dots$ (Abel used $a,a',a'',\dots$).   Thus $\theta(x,y,\ua) \in \K[x,y,\ua]$.  For simplicity, we often write this as $\theta(x,y)$ or $\theta(y)$.  

In Section 2 of \cite{A1}, Abel states his first main result, which concerns the differential
\begin{equation}
\label{dvdef}
dv = \sum_{i=1}^\mu f(x_i,y_i)\, dx_i
\end{equation}
and the corresponding sum of integrals
\[
v = \sum_{i=1}^\mu \int \! f(x_i,y_i)\, dx_i
\]
(see \cite[(8) on p. 148 and (12) on p.\ 149]{A1}).  The sum is over all solutions 
\[
(x_1,y_1),\dots,(x_\mu,y_\mu)
\] 
of $\chi(x,y) = \theta(x,y,\ua) = 0$ for which $dx_i \ne 0$.  We will see in Section \ref{APW} that $x_i$ is an algebraic function of $\ua = a_1,a_2,\dots$, so that $dx_i$ is defined.  Thus the sum is over solutions where $x_i$ is a nonconstant function of $\ua$ and hence transcendental over $\K$.   

Abel claims the following: 
\begin{itemize}
\item[(1)] $dv$ is a rational differential of $\ua$ \cite[p.\ 148]{A1}.
\item[(2)] $v$ is a ``fonction alg\'ebrique et logarithmique'' of $\ua$ \cite[p.\ 149]{A1}.   
\end{itemize}
He gives two proofs that $dv$ is rational in \cite{A1}. The first uses the theory of symmetric functions and is sketched on \cite[p.\ 148]{A1}, with more details in \cite{A4}.  

Abel's second proof for $dv$ is completely different and is a major focus of these notes.  The proof occupies Section 4 of his memoir.  Abel begins as follows:
\begin{quote}
We showed above how one can always form the rational differential $dv$; but as the indicated method will be in general very long, and for functions a little complicated, almost impractical, I am going to give another, by which one will obtain immediately the expression of the function $v$ in all possible cases. \cite[p.\ 150]{A1}
\end{quote}
In equation (36) of \cite[p.\ 158]{A1}, he derives the astonishing formula for $dv$:\footnote{\label{f3} In Abel's collected works \cite[Vol.\ 2, pp.\ 295--296]{A5}, Sylow observes that the formula \eqref{Abel36} is not quite correct. See Sections \ref{SecRoweAbel} and \ref{rewrite} for more details.} 
\begin{equation}
\label{Abel36}
dv = -\text{\large$\varPi$} \frac{F_2x}{\theta_1x} \sum \frac{f_1(x,y)}{\chi'y} \frac{\delta \theta y}{\theta y}
+ \sum{}\rule{0pt}{12pt}'\nu \frac{d^{\nu-1}}{dx^{\nu-1}} \bigg\{\frac{F_2x}{\theta_1^{(\nu)}x} \sum \frac{f_1(x,y)}{\chi'y} \frac{\delta \theta y}{\theta y}\bigg\}.
\end{equation}
There is a \emph{lot} of notation going on here, all of which will be eventually explained. The functions $F_2(x), \theta_1(x), \theta^{(\nu)}_1(x), f_1(x,y)$ that appear in \eqref{Abel36} are polynomials with coefficients in $\overline{\K}$, the algebraic closure of $\K$.  Also, $\chi'y = \frac{\partial}{\partial y} \chi(x,y)$ and $\delta \theta y$ is the differential of $\theta(x,y,\ua)$ with respect to the indeterminates $\ua = a_1,a_2,\dots$, i.e.,
\begin{equation}
\label{deltathetadef}
\delta \theta y = \delta \theta(x, y,\ua) = \tfrac{\partial\theta}{\partial a_1}(x, y,\ua)\,da_1 +  \tfrac{\partial\theta}{\partial a_2}(x, y,\ua)\,da_2 + \cdots. 
\end{equation}
Once we understand the symbol {\large $\varPi$} and the sums involved in \eqref{Abel36}, we will see that $dv$ is indeed a rational differential in $\ua$.

The only place where the indeterminates $\ua$ appear in \eqref{Abel36} is in $\theta y$ and $\delta\theta y$, which leads to an integrated version of the formula:
\begin{equation}
\label{Abel37}
v = C-\text{\large$\varPi$} \frac{F_2x}{\theta_1x} \sum \frac{f_1(x,y)}{\chi'y}  \log \theta y
+ \sum{}\rule{0pt}{12pt}'\nu \frac{d^{\nu-1}}{dx^{\nu-1}} \bigg\{\frac{F_2x}{\theta_1^{(\nu)}x} \sum \frac{f_1(x,y)}{\chi'y}  \log \theta y\bigg\},
\end{equation}
where $C$ is a constant of integration (see \cite[(37) on p.\ 159]{A1}).  This is ``the expression of the function $v$'' promised by Abel in the above quote.

After proving \eqref{Abel37}, Abel asks when $v$ reduces to a constant. This puts conditions on the polynomial $f_1(x,y)$.  He defines $\gamma$ to be the ``number of arbitrary constants'' in $f_1(x,y)$ and gives the following formula for $\gamma$ \cite[(62) on p.\ 168]{A1}:
\begin{equation}
\label{Abel62}
\gamma = \begin{cases} 
\displaystyle n'\mu' \Big( \frac{n' m' -1}2\Big) + n''\mu''\Big( n' m' +  \frac{n'' m''-1}2\Big)\ + &\\[5pt]
\displaystyle n'''\mu'''\Big( n' m' + n'' m'' +  \frac{n''' m'''-1}2\Big) + \cdots  + &\\[5pt]
\displaystyle n^{(\varepsilon)}\mu^{(\varepsilon)}\Big( n' m' + \cdots + n^{(\varepsilon-1)} m^{(\varepsilon - 1)} +  \frac{n^{(\varepsilon)} m^{(\varepsilon-1)}}2\Big)&\\[5pt]
\displaystyle -\ \frac{n'(m'+1)}2 - \cdots - \frac{n^{(\varepsilon)} (m^{(\varepsilon)} +1)}2 + 1.
\end{cases}
\end{equation}
Again, there is more notation to untangle. From a modern point of view, this formula for $\gamma$ counts lattice points in the interior of the Newton polygon of $\chi(x,y)$ and, as we will explain carefully in Section \ref{AbelHolo}, is  closely related to the genus.

\medskip

Equations \eqref{Abel36}, \eqref{Abel37}, \eqref{Abel62} are the  ``three formulas'' in the title of these notes.

\medskip

After the Paris Academy finally published Abel's memoir in 1841,\footnote{The fascinating story of Abel's manuscript---including the way it was handled by the Paris Academy and why it wasn't in his 1839 collected works---is described in the biographies of Abel by Ore \cite[Chapter 20]{Ore} and Stubhaug \cite[pp.\ 548--552]{Stub}. See also Del Centina \cite[pp.\ 87--103]{DelC}.}  various people tried to understand Abel's proofs,
including Boole  \cite{B} in 1857, Rowe \cite{R} in 1881, Brill and Noether \cite{BN} in 1894, Baker \cite{BakerBook} in 1897, and Forsyth \cite{F} in 1918.   Modern treatments of Abel's Theorem that mention \eqref{Abel36} and \eqref{Abel37} include  Houzel \cite{Houzel} and Kleiman \cite{K1}, and Cooke \cite{Cooke} considers the hyperelliptic case treated by Abel \cite{A3} in 1828. The last three references discuss many other articles about Abel's Theorem.

One thing to keep in mind is that Abel's Paris memoir contains several results that share the name ``Abel's Theorem.''  
The version involving \eqref{Abel36} and \eqref{Abel37} is what Kleiman calls \emph{Abel's Elementary Function Theorem} in \cite{K1}.  Other versions of Abel's Theorem will appear briefly in Section \ref{OtherAbelThms} and are described nicely in \cite{K1}.  In many ways, these other versions constitute the real depth of Abel's memoir and explain why people studied it so carefully.  Nevertheless, our hope  is to convince you that Abel's formulas \eqref{Abel36} and \eqref{Abel37}, along with \eqref{Abel62}, represent some significant mathematics in terms of both how they are proved and what they imply.

\section{Our Approach to Abel's Formulas}
\label{SecApproach}

Formulas \eqref{Abel36} and \eqref{Abel37} are a bit overwhelming at first glance.  For the modern reader, however, they are relatively easy to understand in terms of residues, as will be explained in Sections \ref{AMV} and \ref{AvMV}.  One goal of these notes is to prove these formulas and understand what they mean in terms of symbolic computation and residues.    

But rather than just jump into the proof, let us say a bit more about how people have reacted to Abel's memoir.
We begin with two quotes from the papers of Boole and Rowe:
\begin{quote}
{\small\scshape Boole}:\ As presented in the writings of {\scshape Abel}, and of those who immediately followed in his steps, the doctrine of the comparison of transcendents is repulsive from the complexity of the formul\ae\ in which its general conclusions are embodied. \cite[p.\ 746]{B}

\smallskip

\noindent {\small\scshape Rowe}:\ The generality and power of this memoir are well known, but its form is not attractive. \cite[p.\ 713]{R}

\end{quote}
One reason for saying ``repulsive" and ``not attractive'' is that Abel's memoir begins with
\eqref{Abel36} and \eqref{Abel37}.  The complexity of these formulas is indeed overwhelming, not to mention \eqref{Abel62} and the many other formulas that appear in \cite{A1}.  Abel's memoir is not easy reading.

For some commentators, another challenge of \cite{A1} is that while it implicitly involves key ideas in the theory of algebraic curves, there ``is no geometry in Abel’s work'' in the words of Kleiman \cite[p.\ 403]{K1}.  Although contour integrals are hinted at in \cite{A1} (see Section \ref{Limits}), his methods are almost entirely algebraic.
In his 1857 treatise \emph{Theorie der Abel’schen Functionen} \cite{Riemann}, Riemann studied the integrals \eqref{AbInt} from the point of view of Riemann surfaces, though as noted by Edwards, Riemann's approach 
\begin{quote}
\dots\ gave no insight into Abel's \emph{formulas} which, for Abel at least, represented the true substance of the memoir.  After Riemann, the theory of Abelian functions became firmly embedded within the burgeoning theory of functions of a complex variable. \cite[p.\ 202]{Edwards}
\end{quote}
For us, things like contour integrals and Riemann surfaces are less relevant since Abel's approach is closer to symbolic computation than to geometry or complex analysis.  

Abel's first proof that $dv$ is a rational differential uses the theory of symmetric functions.  Although Dieudonn\'e admires Abel's result, he also says that the proof is “hardly more than an exercise in the theory of the symmetric functions of the roots of a polynomial” \cite[p.\ 20]{Dieudonne}.

These notes will argue that the symbolic aspects of Abel's formulas for $dv$ and $v$ go much deeper.   In the sections that follow, we will prove modern versions of Abel's formulas \eqref{Abel36} and \eqref{Abel37} in a way that emphasizes the role of symbolic computation.  In fact, one way to think about them is that they are not just formulas; rather, they are \emph{symbolic algorithms embedded within formulas}.

Although we will not provide detailed implementations of the algorithms involved (that is not the goal of these notes), Sections \ref{Abeldv} and \ref{Abelv} will present complete proofs of modern versions \eqref{Abel36} and \eqref{Abel37}, along with informal descriptions of the algorithms and examples of what the computations look like in practice.  
Section~\ref{AbelHolo} will describe how Abel used his formulas to study questions related to holomorphic differentials and the genus, though he did not think in these terms.  We will give a complete proof of \eqref{Abel62} and describe its relation to counting lattice points and the genus.  

Section \ref{OtherAbelThms} will explain how Abel's memoir was inspired by the classic addition theorems for elliptic integrals. 
His generalizations of these results are encapsulated in versions of Abel's Theorem in which the genus plays a prominent role.  Finally, Appendix~\ref{AbelPf} is for readers interested in how Abel proved  \eqref{Abel36} and \eqref{Abel37} and how Boole and Rowe used residues to recast his proof into a more understandable form.

This will not be easy---some hard work will be required---but at the end, you will see that Abel's complicated formulas \eqref{Abel36}, \eqref{Abel37}, and \eqref{Abel62} are both beautiful and fully understandable.   As we approach the 200th anniversary of the submission of his memoir \cite{A1} to the Paris Academy, it is important to give Abel's formulas the respect they deserve.



\section{Abel's Formula for \emph{dv}}
\label{Abeldv}

The goal of this section is to state and prove a modern version of Abel's formula for $dv = \sum_{i=1}^\mu f(x_i,y_i)\, dx_i$. We also give examples of how to use the formula.

\subsection{A Modern Version of Abel's Formula for \emph{dv}}
\label{AMV}

Abel's original version of his formula \eqref{Abel36} for $dv$ uses lots of notation and states that
\begin{equation}
\label{ATModern0}
dv = -\text{\large$\varPi$} \frac{F_2x}{\theta_1x} \sum \frac{f_1(x,y)}{\chi'y} \frac{\delta \theta y}{\theta y}
+ \sum{}\rule{0pt}{12pt}'\nu \frac{d^{\nu-1}}{dx^{\nu-1}} \bigg\{\frac{F_2x}{\theta_1^{(\nu)}x} \sum \frac{f_1(x,y)}{\chi'y} \frac{\delta \theta y}{\theta y}\bigg\}.
\end{equation}
Recast in modern terms, Abel's formula expresses $dv$ as a sum of residues
\begin{equation}
\label{ATModern}
dv = \sum_{j=1}^\alpha \res_{\beta_j}\hskip-1pt \Big(\sum_{\ell=1}^n f(x,y^{(\ell)}) \hskip1pt\frac{\delta \theta(y^{(\ell)})}{\theta(y^{(\ell)})}\Big) -  \text{\large$\varPi$}\Big( \sum_{\ell=1}^n f(x,y^{(\ell)}) \hskip1pt\frac{\delta \theta(y^{(\ell)})}{\theta(y^{(\ell)})} \Big),
\end{equation}
where  we find it more convenient to put the {\large$\varPi$} term at the end.  In \eqref{ATModern}, the $y^{(\ell)}$ are the roots of $\chi(x,y)$, the operator ${-}\text{\large$\varPi$}$ can be regarded as a residue at $\infty$, and the $\beta_j$ are the roots of an explicitly constructed polynomial $f_2(x)F_0(x)$.  All of this will be explained in Section \ref{APW}.  Recall that $\delta\theta(y^{(\ell)})$ was defined earlier in \eqref{deltathetadef}.

Note that the same sum over the roots $y^{(\ell)}$ appears twice in \eqref{ATModern}.  By the theory of symmetric functions, this sum is rational in $x$, which we write as
\begin{equation}
\label{Sxdef}
S(x) =  \sum_{\ell=1}^n f(x,y^{(\ell)})\hskip1pt \frac{\delta \theta(y^{(\ell)})}{\theta(y^{(\ell)})}.
\end{equation}
The proof of Lemma \ref{AbelLem} will describe Abel's clever method for computing $S(x)$.  Using $S(x)$, we can rewrite \eqref{ATModern} as
\begin{equation}
\label{ATModernS}
dv = \sum_{j=1}^\alpha \res_{\beta_j}\hskip-1pt(S(x))  - \text{\large$\varPi$}(S(x)).
\end{equation}
We will prove  \eqref{ATModern} and \eqref{ATModernS} in Section \ref{APf}.
The examples in Section \ref{AEx} will show that Abel's formula  \eqref{ATModernS} is easy to understand and
straightforward to compute.

However, getting from Abel's original formula \eqref{ATModern0} to the modern version \eqref{ATModern}  takes some work.  There is a \emph{lot} of notation to unwind, as will be explained in Section \ref{rewrite}.

\subsection{Preliminary Work}
\label{APW}
Here we explain the notation needed to understand \eqref{ATModern} and prove some preliminary results needed for the proof.

Recall that $\chi(y) = \chi(x,y) \in \K[x,y]$ and $\theta(y) = \theta(x,y,\ua) \in \K[x,y,\ua]$, where $\ua =a_1, a_2, \dots$.   Set $\chi'(y) = \frac{\partial\chi}{\partial y}(x,y)$.  In \eqref{ATModern} and \eqref{Sxdef}, we have the roots $y^{(1)},\dots,y^{(n)}$ of $\chi(x,y)$, where $\chi(x,y)$ is regarded as a polynomial in $y$ with coefficients in $\K[x]$.  
These roots can be constructed by standard methods of symbolic computation (see, for example, \cite[Chapter 5]{Mora}), meaning that they are computable elements of a finite extension of $\K(x)$ that is a splitting field of the irreducible polynomial $\chi(x,y) \in\K(x)[y]$.  

\subsubsection*{Solutions of $\chi(x,y) =\theta(x,y,\underline{a}) = 0$}  To find the solutions, we first eliminate $y$ using a resultant.
As above, $y^{(1)},\dots,y^{(n)}$ are the roots of $\chi(x,y) \in \K[x,y]$.  Since $\chi(x,y)$ is monic in $y$, the Poisson formula for the resultant (see \cite[(1.4)]{UAG}) implies that
\[
\mathrm{Res}(\chi(x,y),\theta(x,y,\ua),y) = \prod_{\ell=1}^n \theta(x,y^{(\ell)},\ua) \in \K[x,\ua].
\]
Following Abel \cite[(3) on p.\ 147]{A1}, we write this as
\[
r(x) = r(x,\ua) =  \prod_{\ell=1}^n \theta(x,y^{(\ell)}).
\]

Since $\chi(x,y)$ is monic in $y$, properties of the resultant imply that the roots of $r(x)$ are precisely the $x$-coordinates of solutions of the system $\chi(x,y) = \theta(x,y) = 0$.   Thus, if $(x_0,y_0)$ is a solution of the system, then $r(x_0,\ua) = 0$, so that $x_0$ is an algebraic function of $\ua$.  
Furthermore, all roots of $r(x)$ arise in this way.

However, recall that the sum $\sum_{i=1}^\mu f(x_i,y_i)\,dx_i$ uses only solutions $(x_i,y_i)$ for which 
$x_i$ is transcendental over $\K$.   This leads to the factorization
\begin{equation}
\label{Efact}
r(x,\ua) = F_0(x) F(x,\ua),
\end{equation}
where all roots of $F_0(x)$ (resp.\ no roots of $F(x) = F(x,\ua)$) are algebraic over $\K$. This gives the polynomial $F_0(x)$ that appears in the discussion following \eqref{ATModern}.  Note that the $x_i$'s appearing in $dv = \sum_{i=1}^\mu f(x_i,y_i)\, dx_i$ are precisely the roots of $F(x)$.

 Here is an example where $F_0(x)$ has positive degree.

\begin{example}
\label{Rowe718Ex}
For the field $\K = \Q$, let $\chi(x,y) = y^2 - (1 + x^4)$ and $\theta(x,y,a_1,a_2) = y - (1+a_1 x + a_2 x^2)$.  Then $y^{(1)}, y^{(2)} = \pm \sqrt{1+x^4}$, so that 
\begin{align*}
r(x,a_1,a_2) &= \theta(x,y^{(1)},a_1,a_2) \theta(x,y^{(2)},a_1,a_2) = (1+a_1 x + a_2 x^2)^2 - (1+x^4)\\
&= x\big(2 a_1 + (a_1^2+2a_2) x + 2 a_1 a_2 x^2 + (a_2^2-1) x^3\big).
\end{align*}
Thus
\[
F_0(x) = x,\ 
F(x,a_1,a_2) = (a_2^2 -1)x^3 + 2a_1a_2 x^2 + (a_1^2 + 2a_2)x + 2a_1.\eqno\Diamond
\]
\end{example}

Abel knew that $F_0(x)$ can acquire roots when the indeterminates $\ua = a_1,a_2,\dots$ are allowed to satisfy linear relations.  Rowe's 1881 paper on Abel's formula includes the following example  \cite[p.\ 718]{R}.

\begin{example}
\label{Rowe718Ex2}
Let $\K$, $\chi(x,y)$ and $\theta(x,y,a_1,a_2)$ be as in Example \ref{Rowe718Ex}, and assume that
 $(1,\sqrt2)$ is a solution of $\chi(x,y) = \theta(x,y,a_1,a_2) = 0$.  Then $\theta(1,\sqrt2,a_1,a_2) = 0$, which implies that
\[
a_1 + a_2 = \sqrt2 - 1.
\]
Also, $r(1,a_1,a_2) = 0$ since $(1,\sqrt2)$ is a solution.  This implies that $r(x,a_1,a_2) = F_0(x) F(x,a_1,a_2)$, where
\[
F_0(x) = x(x-1), \
F(x,a_1,a_2) = (a_2^2 -1)x^2 + (2a_1a_2 + a_2^2 -1)x -2a_1.
\]
(This corrects some minor errors in \cite[p.\ 718]{R}.)\hfill$\Diamond$
\end{example}

Up to now, we have regarded the $a_i$ as being algebraically independent.  To allow for the behavior illustrated in Example \ref{Rowe718Ex2}, we will assume that a subset of $\ua = a_1,a_2,\dots$ is algebraically independent and that the remaining $a_i$'s are linearly related to the independent ones.  These relations may require a finite extension of $\K$, as  in Example \ref{Rowe718Ex2}.  Replacing $\K$ with a suitable finite extension, we may assume that the relations have coefficients in $\K$ and that $F_0(x) \in \K[x]$, $F(x) \in \K[x,\ua]$.

We next consider $F(x) = F(x,\ua)$ in more detail.  Since the roots of $r(x)$ are the $x$-coordinates of the solutions of  $\chi(x,y) = \theta(x,y) = 0$, the definition of $F(x)$ shows that its roots are $x_1,\dots,x_\mu$, where $(x_i,y_i)$, $i = 1,\dots,\mu$, are the solutions with $x_i$ transcendental over $\K$.  We will make the following important assumption
\begin{equation}
\label{Assume1}
\text{$F(x)$ and $F'(x)$ have no common factors, i.e., $F'(x_i) \ne 0$ for $i = 1,\dots,\mu$}.
\end{equation}
(Here, $F'(x) = \frac{\partial}{\partial x} F(x,\ua)$.) This  is needed for the proof given in Section \ref{APf}.  

Abel implicitly assumes \eqref{Assume1} at the beginning of his proof since $F'(x_i)$ appears in the denominator in multiple places \cite[p.\ 151]{A1}.\footnote{In Abel's 1828 Crelle paper \cite{A3} on hyperelliptic Abelian integrals (where he mentions his Paris memoir), his {\it Th\'eor\`eme} II allows $F(x)$ to have multiple roots, though no proof is given.}   
Rowe explicitly assumes  \eqref{Assume1} in his 1881 treatment of Abel's formula \cite[p.\ 720]{R}.

Assumption \eqref{Assume1} has the following nice consequence.

\begin{proposition}
\label{yjunique}
Let $(x_i,y_i)$, $i = 1,\dots,\mu$, be as above.  If \eqref{Assume1} holds, then 
\[
y-y_i = \gcd(\chi(x_i,y), \theta(x_i,y)), \quad i = 1,\dots,\mu.
\]
Furthermore, $x_1,\dots,x_\mu$ are distinct, and for all $i$, $y_i$ can be expressed as a rational function in $x_i$ with coefficients in $\K(\ua)$.
\end{proposition}

\begin{proof}
We first show that  $x_1,\dots,x_\mu$ are distinct.
Let $m_i \ge 1$ be the multiplicity of $(x_i,y_i)$ as a solution of $\chi(x,y) = \theta(x,y) = 0$.  Since $\chi(x,y)$ is monic in $y$, Theorem 1.1 of \cite{DC} implies that the resultant $r(x)$ is, up to a constant, the product of $(x-\gamma_1)^{m_\gamma}$ for all solutions $\gamma = (\gamma_1,\gamma_2)$ of the system, where $m_\gamma$ is again the multiplicity ($\chi(x,y)$ being monic in $y$ guarantees that there are no solutions at $\infty$ in the sense of \cite{DC}).
Since $F(x)$ collects the factors of $r(x)$ where $\gamma_1$ is transcendental over $\K$, it follows that
\[
F(x) = c \prod_{i=1}^\mu (x-x_i)^{m_i}
\]
for some constant $c$.  If $x_i = x_j$ for $i\ne j$, $F(x)$ has a root of multiplicity $m_i + m_j \ge 2$, which contradicts \eqref{Assume1}.  Thus $x_1,\dots,x_\mu$ must be distinct.
 
We now compute the $\gcd$.   Let $y_i = y_i^{(1)},\dots, y_i^{(n)}$ be the roots of $\chi(x_i,y) = 0$.  The $y_i^{(\ell)}$ are distinct since the discriminant of $\chi(x,y) \in \K[x,y]$ with respect to $y$ lies in $\K[x]$ and $x_i$ is transcendental over $\K$.  If $\theta(x_i,y_i^{(\ell)}) = 0$ for some $\ell \ge 2$, then since $\chi(x_i,y_i^{(\ell)}) = 0$, we would have solutions $(x_i,y_i) \ne (x_i,y_i^{(\ell)})$ of $\chi(x,y) = \theta(x,y) = 0$.  But we  showed above that the solutions have distinct $x$-coordinates.  Thus $\theta(x_i,y_i^{(\ell)}) \ne 0$ for $\ell \ge 2$, which implies
\[
\gcd(\chi(x_i,y),\theta(x_i,y)) = \gcd((y-y_i^{(1)})\cdots (y-y_i^{(n)}),\theta(x_i,y)) = y-y_i^{(1)} = y-y_i
\]
since the factors $y-y_i^{(1)},\dots y-y_i^{(n)}$ are distinct and only the first divides $\theta(x_i,y)$.  The final assertion of the proposition follows by computing the $\gcd$ via the Euclidean Algorithm over the field $\K(x_i,\ua)$ (remember that $\theta(x_i,y) \in \K[x_i,y,\ua]$).
\end{proof}

Abel \cite[p.\ 148]{A1} says that the value of $y_i$ can be expressed as a rational function of $x_i$ and $\ua$.  According to Cooke \cite[pp.\ 407--408]{Cooke}, Abel was using the ``important principle'' that when $y_i$ is the only solution of $\chi(x_i,y) = \theta(x_i,y) =0$, the gcd of $\chi(x_i,y)$ and $\theta(x_i,y)$ gives the desired representation of $y_i$.  Proposition \ref{yjunique} justifies this principle in our situation.

\subsubsection*{The Rational Function $f(x,y)$} Next consider the rational function $f(x,y)$ that appears in $dv = \sum_{i=1}^\mu f(x_i,y_i)\, dx_i$.  Since $\chi(x_i,y_i) = 0$, we only need $f(x,y)$ modulo $\chi(x,y)$.  Following Abel, we represent $f(x,y)$ in the form
\begin{equation}
\label{fxyform}
f(x,y) = \frac{f_1(x,y)}{f_2(x) \chi'(y)} 
\end{equation}
for polynomials $f_1(x,y)$ and $f_2(x)$ with coefficients in $\K$.   This defines the polynomial $f_2(x)$ mentioned in the discussion following \eqref{ATModern}.  The presence of $\chi'(y)$ in the denominator will be useful.  

Abel proves \eqref{fxyform} in \cite[pp.\ 151--152]{A1}.  For a modern proof, consider the extension $\K(x) \subseteq  \K(x)[y]/\langle \chi(x,y)\rangle$ of degree $n$.  Then, regarding $\chi'(y) f(x,y)$ as an element of the larger field, we can write 
\[
\chi'(y) f(x,y) = b_0(x) + b_1(x) y + \cdots + b_{n-1}(x) y^{n-1},\quad b_0(x),\dots,b_{n-1}(x) \in \K(x).
\]
The desired representation follows after dividing by $\chi'(y)$ and letting $f_2(x) \in \K[x]$ be the least common denominator of the rational functions $b_0(x),\dots,b_{n-1}(x)$.  

\subsubsection*{Differentials and Differential Fields}
Abel and his contemporaries used informal notions of differentials.  We will use the modern language of differential forms and differential fields to make this precise.  

For our computable field $\K$, we have the field of rational functions $\K(x,\ua) = \K(x, a_1, a_2, \dots)$, where $h \in \K(x,\ua)$ gives the rational $1$-form
\[
dh = \frac{\partial h}{\partial x} \hskip1pt dx+ \frac{\partial h}{\partial a_1} \hskip1pt da_1 +\frac{\partial h}{\partial a_2} \hskip1pt da_2 + \cdots. 
\]
By the theory of differential fields (see \cite{Bronstein}), the commuting derivations $\frac{\partial}{\partial x}$ and $\frac{\partial}{\partial a_i}$ extend uniquely to any algebraic extension of $\K(x,\ua)$ by \cite[Theorem 3.2.2]{Bronstein} and still commute.  For example, if $(x_0,y_0)$ is a solution of $\chi(x,y) = \theta(x,y,\ua) = 0$, then $x_0$ is algebraic over $\K(\ua)$ by the paragraph preceding \eqref{Efact}, so that $\frac{\partial x_0}{\partial a_i}$ and hence $dx_0$ are all defined.  In particular, we now have a precise meaning for Abel's differential form
\[
dv = \sum_{i=1}^\mu f(x_i,y_i)\, dx_i
\]
from \eqref{dvdef}.  We will give an explicit formula for $dx_i$ in \eqref{dxj}.  

\subsubsection*{Residues}
We recall some standard facts about residues and define the operator {\large$\varPi$} used in \eqref{ATModern0} and \eqref{ATModern}.   We begin with residues.  As is well known, if $g(x)$ has a meromorphic Laurent series expansion in powers of $x - \beta$, say
\begin{equation}
\label{gxexp}
g(x) = \frac{a_{-M}}{(x-\beta)^M} + \cdots + \frac{a_{-1}}{x-\beta} + a_0 + a_1 (x-\beta) + a_2 (x-\beta)^2 +\cdots,
\end{equation}
then
\begin{equation}
\label{resbetadef}
\res_\beta(g(x)) = a_{-1} = \text{coefficient of $\frac{1}{x-\beta}$ in the expansion at $x = \beta$}. 
\end{equation}
Turning to $x = \infty$, suppose that $g(x)$ has a meromorphic Laurent series expansion in powers of $u=1/x$, say
\begin{align*}
g(x) = g(1/u) &= \frac{b_{-M}}{u^M} + \cdots  + \frac{b_{-1}}{u} + b_0 + b_1u + b_2 u^2 + \cdots\\
&= b_{-M} x^M + \cdots + b_{-1}x + b_0 + \frac{b_1}{x} + \frac{b_2}{x^2} + \cdots.
\end{align*}
This is an expansion in descending powers of $x$.  Following Abel \cite[p.\ 155]{A1}, we define
\begin{equation}
\label{C1xdef}
\text{\large$\varPi$}(g(x)) = b_1 = \text{coefficient of $\dfrac1x$ in the expansion of $g(x)$ at $x = \infty$.}
\end{equation}

In terms of differential forms, let $\omega = g(x)\,dx$.  For $\beta$ as above, 
\[
\res_\beta(\omega) = \res_\beta(g(x)) = a_{-1} 
\]
since $x-\beta$ is a local coordinate at $x = \beta$ and $d(x-\beta) = dx$.  On the other hand, $u = 1/x$ is a local coordinate at $x =\infty$, so that 
\[
\omega = g(1/u)\, d(1/u) = g(1/u)\cdot \frac{-1}{\ u^2}\, du.
\]
Hence
\[
\res_\infty(\omega) = \text{coefficient of $u^{-1}$ in the expansion of $-u^{-2}g(1/u)$},
\]
which easily seen to equal $-b_1$ in the above expansion of $g(x)$ in $u = 1/x$. Thus 
\begin{equation}
\label{PiRes}
\res_\infty(\omega) = \res_\infty(g(x)\,dx) = -b_1 = -\text{\large$\varPi$}(g(x)).
\end{equation}
This justifies our comment that ${-}\text{\large$\varPi$}$ can be regarded as a residue at $\infty$.

We  next recall  some well-known methods for computing residues.  First consider $g(x)$ as in \eqref{gxexp} and let $\nu \ge M$.  Then
\begin{equation}
\label{computeres}
\res_\beta(g(x)) = \frac1{(\nu-1)!} \frac{d^{\nu-1}}{dx^{\nu-1}}\big( (x-\beta)^\nu g(x)\big)\Big|_{x=\beta}.
\end{equation}
This formula will enable us to interpret various terms in \eqref{Abel36} as residues, though Abel did not think in terms of residues.   Notice that in \eqref{computeres}, $\nu$ is greater than or equal to the order of $\beta$ as a pole of $g(x)$.

Another useful formula is that if $h(x)$ is a polynomial with $\beta$ as a simple root and $g(x)$ has an expansion in nonnegative powers of $x-\beta$, then
\begin{equation}
\label{computeres1}
\begin{aligned}
\res_\beta\Big(\frac{g(x)}{h(x)}\Big) &= \res_\beta\Big(\frac{g(\beta) + g'(\beta)(x-\beta) + \cdots}{h'(\beta)(x-\beta) + \cdots}\Big)\\ &=
 \res_\beta\Big(\frac{g(\beta)}{h'(\beta)} \cdot \frac1{x-\beta} + \cdots\Big) =
\frac{g(\beta)}{h'(\beta)}.
\end{aligned}
\end{equation}

\subsubsection*{The Global Residue Theorem for $\PP^1$} This important result states that for any rational function $\varphi(x)$, the sum of the residues of the differential form $\varphi(x)\,dx$ is zero, i.e., 
\begin{equation}
\label{GRT}
\sum_{\beta \in \PP^1} \res_\beta(\varphi(x)\,dx) = 0.
\end{equation}
Serre gives simple proof using partial fractions in \cite[Lemma 3 on p.\ 21]{Serre}. The Global Residue Theorem will play a key role in our proof of Abel's formula.  Here is an easy corollary of \eqref{GRT}.

\begin{corollary}
\label{BooleCor}
If $H(x) = a_m x^m + \cdots$ is a nonconstant polynomial with simple roots $x_1,\dots,x_m$
and $H_1(x)$ is any polynomial of degree $\le m-2$, then
\[
\sum_{j=1}^m \frac{x_j^{m-1}}{H'(x_j)} = \frac{1}{a_m}  \ \,\text{and}\  \,\, \sum_{j=1}^m \frac{H_1(x_j)}{H'(x_j)} = 0.
\]
\end{corollary}

\begin{proof}
The only possible poles of the differential form $\omega = \dfrac{x^{m-1}}{H(x)}\hskip.5pt dx$ are $x_1,\dots,x_m$ and $\infty$, so by \eqref{GRT},
\[
0 = \sum_{j=1}^m \res_{x_j}(\omega) + \res_\infty(\omega) = \sum_{j=1}^m \frac{x_j^{m-1}}{H'(x_j)} - \text{\large$\varPi$}\Big(\frac{x^{m-1}}{H(x)}\Big),
 \]
where the second equality uses \eqref{computeres1} and \eqref{PiRes}.  Then we are done since by the definition of {\large$\varPi$},
\[
\text{\large$\varPi$}\Big(\frac{x^{m-1}}{H(x)}\Big) = \text{\large$\varPi$}\Big(\frac{x^{m-1}}{a_m x^m + \cdots }\Big) = \frac{1}{a_m}.
\]
The proof of the second assertion is equally easy and hence omitted.
\end{proof}

The formulas in Corollary \ref{BooleCor} were well known in Abel's time.  He calls them ``formules connues''  \cite[p.\ 153]{A1}.   We will say more about Corollary \ref{BooleCor} in Section~\ref{APf}.

\subsection{Proof of Abel's Formula for \emph{dv}}
\label{APf}
Here is what we have so far:
\begin{itemize}
\item The polynomials $\chi(x,y) \in \K[x,y]$ and $\theta(x,y,\ua) \in \K[x,y,\ua]$.
\item The roots $y^{(\ell)}$, $\ell = 1,\dots,n$, of $\chi(x,y)$.
\item The solutions $(x_i,y_i)$, $i = 1,\dots,\,\mu$, of $\chi(x,y) = \theta(x,y,\ua) = 0$ with $dx_i \ne 0$.
\item The resultant $r(x) = \prod_{\ell=1}^n \theta(x,y^{(\ell)},\ua) = F_0(x) F(x,\ua)$, where $F_0(x) \in \K[x]$ and $F(x) = F(x,\ua) \in \K[x,\ua]$.
\item The rational function $f(x,y) = \dfrac{f_1(x,y)}{f_2(x)\chi'(y)}$.
\item The residue operators $\res_\beta$ \strut and {\large$\varPi$} from \eqref{resbetadef} and \eqref{C1xdef}.
\end{itemize}
The goal is to prove the version of Abel's formula for $dv = \sum_{i=1}^n f(x_i,y_i)\,dx_i$ stated in \eqref{ATModernS}, namely
\[
dv = \sum_{j=1}^\alpha \res_{\beta_j}\hskip-1pt(S(x)) - \text{\large$\varPi$}(S(x)),
\]
where $\beta_1,\dots,\beta_\alpha$ are the roots of $f_2(x)F_0(x)$ and
\[
S(x) =  \sum_{\ell=1}^n f(x,y^{(\ell)})\hskip1pt \frac{\delta \theta(y^{(\ell)})}{\theta(y^{(\ell)})}.
\]

The first step of the proof is to interpret $f(x_i,y_i)\,dx_i$ as a residue.

\begin{lemma}
\label{fxiyiSxi}
With the above setup, assume \eqref{Assume1}.   Then
\[
f(x_i,y_i)\,dx_i = -\res_{x_i}(S(x)),\quad i = 1,\dots,\mu.
\]
\end{lemma}

\begin{remark}
For $\res_{x_i}(S(x))$, we regard the $da_i$ appearing in $\delta \theta(y^{(\ell)})$ as new variables, so that $\res_{x_i}(S(x))$ is computed by regarding $S(x)$ as a rational function.  
\end{remark}

\begin{proof}
We follow Abel, who begins by noting that $F(x_i,\ua) = 0$ implies 
\[
0 = d F(x_i,\ua) = F'(x_i,\ua) \,dx_i + \delta F(x_i,\ua),
\]
which can be written more simply as
\[
0 =  F'(x_i) \,dx_i + \delta F(x_i).
\]
Since $r(x) = F_0(x) F(x)$ and $F_0(x) \in \K[x]$ does not involve $\ua$, multiplying by $F_0(x)$ gives
\[
0 = F_0(x_i) F'(x_i) \,dx_i + \delta r(x_i),
\]
so that
\begin{equation}
\label{dxj}
dx_i =  -\frac{\delta r(x_i)}{F_0(x_i) F'(x_i)}.
\end{equation}
Note that $F_0(x_i) \ne 0$ by the definition of $F_0(x)$, and $F'(x_i) \ne 0$ by assumption \eqref{Assume1}.  Thus the denominator is nonvanishing.  

For the numerator of \eqref{dxj}, note that $r(x) = \prod_{\ell=1}^n \theta(x,y^{(\ell)})$ implies
\[
\delta r(x) = \sum_{\ell = 1}^n \Big(\prod_{k\ne\ell}  \theta(x,y^{(k)}) \Big)\hskip1pt \delta\theta(x,y^{(\ell)}).
\]
Evaluating this at $x_i, y_i= y_i^{(1)}, y_i^{(2)},\dots,y_i^{(n)}$ gives
\[
\delta r(x_i) = \sum_{\ell = 1}^n \Big(\prod_{k\ne\ell}  \theta(x_i,y_i^{(k)}) \Big)\hskip1pt \delta\theta(x_i,y_i^{(\ell)})
=\Big(\prod_{k=2}^n  \theta(x_i,y_i^{(k)}) \Big)\hskip1pt \delta \theta(x_i,y_i^{(1)}),
\]
where the second equality follows from the fact that $\theta(x_i,y_i^{(1)}) = \theta(x_i,y_i) =0$ is a factor of $\prod_{k\ne\ell}  \theta(x_i,y_i^{(k)}) $ for $\ell \ge 2$. Hence \eqref{dxj} can be written in the form
\[
dx_i = -\frac{1}{F_0(x_i) F'(x_i)}
\Big(\prod_{k=2}^n  \theta(x_i,y_i^{(k)}) \Big)
\hskip1pt\delta \theta(x_i,y_i^{(1)}).
\]
This is the third display on \cite[p.\ 151]{A1}.  Now multiply both sides by our rational function $f(x_i,y_i) = f(x_i,y_i^{(1)})$ to obtain
\begin{equation}
\label{fxiyidxi}
\begin{aligned}
f(x_i,y_i)\,dx_i &= -\frac{1}{F_0(x_i) F'(x_i)} f(x_i,y_i^{(1)}) 
\Big(\prod_{k=2}^n  \theta(x_i,y_i^{(k)}) \Big)\hskip1pt 
\delta \theta(x_i,y_i^{(1)})\\
&= -\frac{1}{F_0(x_i) F'(x_i)} \sum_{\ell = 1}^n f(x_i,y_i^{(\ell)}) 
\Big(\prod_{k\ne\ell}  \theta(x_i,y_i^{(k)}) \Big)\hskip1pt 
\delta \theta(x_i,y_i^{(\ell)}).
\end{aligned}
\end{equation}
As above, the last line follows since $\prod_{k\ne\ell}  \theta(x_i,y_i^{(k)}) = 0$ for $\ell \ge 2$.  

To relate this to $S(x)$, note that
\begin{equation}
\label{Sxformula}
\begin{aligned}
S(x) &=  \sum_{\ell=1}^n f(x,y^{(\ell)})\hskip1pt \cdot   \frac{\prod_{k\ne\ell}  \theta(x,y^{(k)})}{\prod_{k\ne\ell}  \theta(x,y^{(k)})}  \cdot \frac{\delta \theta(x,y^{(\ell)})}{\theta(x,y^{(\ell)})} \\
&= \frac{1}{F_0(x)F(x)} \sum_{\ell = 1}^n f(x,y^{(\ell)}) \Big(\prod_{k\ne\ell}  \theta(x,y^{(k)}) \Big)\hskip1pt \delta \theta(x,y^{(\ell)})\\
&= \frac{\displaystyle  \frac{1}{F_0(x)} \sum_{\ell = 1}^n f(x,y^{(\ell)}) \Big(\prod_{k\ne\ell}  \theta(x,y^{(k)}) \Big)\hskip1pt \delta \theta(x,y^{(\ell)})}{F(x)}
\end{aligned}
\end{equation}
Since $x_i$ is a simple root of $F(x)$ and the numerator of the above expression for $S(x)$ is defined at $x_i$, 
\eqref{computeres1} implies that
\[
\res_{x_i}(S(x)) = \frac{\displaystyle  \frac{1}{F_0(x_i)} \sum_{\ell = 1}^n f(x_i,y_i^{(\ell)}) \Big(\prod_{k\ne\ell}  \theta(x_i,y_i^{(k)}) \Big)\hskip1pt \delta \theta(x_i,y_i^{(\ell)})}{F'(x_i)} = -f(x_i,y_i)\,dx_i,
\]
where the last equality follows from \eqref{fxiyidxi}.  This proves the lemma.
\end{proof}

Adding up \eqref{fxiyiSxi} for $i = 1,\dots,\mu$, we obtain
\begin{equation}
\label{SumR}
dv = \sum_{i=1}^\mu f(x_i,y_i)\,dx_i = -\sum_{i=1}^\mu \res_{x_i}(S(x)).
\end{equation}
The basic idea is that Abel's formula \eqref{ATModernS} follows from \eqref{SumR} and the Global Residue Theorem.  This requires knowing the poles of $S(x)$, which leads to the next lemma.

\begin{lemma}
\label{AbelLem}
There is a polynomial differential form $R(x)$ such that
\[
S(x) = \frac{R(x)}{f_2(x) F_0(x) F(x)},
\]
where $f_2(x)$ comes from the representation $f(x,y) = \dfrac{f_1(x,y)}{f_2(x)\chi'(y)}$.
\end{lemma}

\begin{proof}
If we combine the second line of \eqref{Sxformula} with the formula for $f(x,y)$ in the statement of the lemma, we obtain
\[
S(x) =  \frac{\displaystyle \sum_{\ell = 1}^n \frac{f_1(x,y^{(\ell)})}{\chi'(y^{(\ell)})}\, \prod_{k\ne \ell}\theta(x,y^{(k)}) \,\delta \theta(x,y^{(\ell)})}{f_2(x) F_0(x) F(x)}.
\]
To analyze this, fix $\ell$ and consider the differential form in $da_1, da_2, \dots$ defined by
\[
\Omega_\ell =   f_1(x,y^{(\ell)})\, \prod_{k\ne \ell}\theta(x,y^{(k)}) \,\delta \theta(x,y^{(\ell)}).
\]
Then
\begin{equation}
\label{SOmega}
S(x) = \frac{\displaystyle \sum_{\ell = 1}^n \frac{\Omega_\ell}{\chi'(y^{(\ell)})}}{f_2(x) F_0(x) F(x)}.
\end{equation}

Note that $\prod_{k\ne \ell}\theta(x,y^{(k)})$ is a polynomial in $\K[x,y^{(\ell)},\ua]$ since $y^{(1)},\dots,y^{(n)}$ are the roots of $\chi(x,y) \in \K[x,y]$, which is monic in $y$.  It follows that $\Omega_\ell$ can be written as a differential form with coefficients in $\K[x,y^{(\ell)},\ua]$, and since $\chi(x,y^{(\ell)}) = 0$ and $\chi(x,y)$ has degree $n$ in $y$, we can assume that $\Omega_\ell$ has degree at most $n-1$ in $y^{(\ell)}$, still with coefficients in $\K[x,y^{(\ell)},\ua]$, again because $\chi(x,y)$ is monic in $y$. Thus we can write
\[
\Omega_\ell = R(x) \big(y^{(\ell)}\big)^{n-1} + R^{(1)}(x,y^{(\ell)}),
\]
where $R(x)$ has coefficients in $\K[x,\ua]$ and $R^{(1)}(x,y^{(\ell)})$ has coefficients in $\K[x,y^{(\ell)}),\ua]$ and has degree $\le n-2$ in $y^{(\ell)}$.  The automorphism taking $y^{(\ell)}$ to $y^{(\ell')}$ shows that $R$ and $R^{(1)}$ do not depend on $\ell$. 

The notation $R$ and $R^{(1)}$ is due to Abel, as is the following analysis of the numerator of \eqref{SOmega} (see \cite[pp.\ 152--153]{A1}):  
\begin{align*}
\sum_{\ell=1}^n \frac{\Omega_\ell}{\chi'(y^{(\ell)})} &= \sum_{\ell=1}^n \frac{R(x) \big(y^{(\ell)}\big)^{n-1} + R^{(1)}(x,y^{(\ell)})}{\chi'(y^{(\ell)})}\\
&=R(x) \sum_{\ell=1}^n \frac{\big(y^{(\ell)}\big)^{n-1}}{\chi'(y^{(\ell)})} + \sum_{\ell=1}^n \frac{R^{(1)}(x,y^{(\ell)})}{\chi'(y^{(\ell)})}.
\end{align*}
Using the standard formulas (``formules connues'' \cite[p.\ 153]{A1})
\begin{equation}
\label{AbelLemPf}
\begin{aligned}
\sum_{\ell=1}^n \frac{\big(y^{(\ell)}\big)^{n-1}}{\chi'(y^{(\ell)})} &= 1\\
\sum_{\ell=1}^n \frac{R^{(1)}(x,y^{(\ell)})}{\chi'(y^{(\ell)})}  &= 0 \ \ (\text{since $R^{(1)}(x,Y)$ has degree $\le n-2$ in $Y$),}
\end{aligned}
\end{equation}
Abel concludes that
\[
\sum_{\ell=1}^n \frac{\Omega_\ell}{\chi'(y^{(\ell)})}= R(x) \cdot 1 + 0 = R(x).
\]
The lemma follows by substituting this into \eqref{SOmega}.
\end{proof} 

\begin{remark}
The  formulas \eqref{AbelLemPf} are proved in Corollary \ref{BooleCor}.   There are two interesting things to note here:
\begin{itemize}
\item The algorithm for computing the global residue of a rational function described in \cite[pp.\ 7--8]{CD} is \emph{precisely} the method used by Abel in 1826.
\item The vanishing result on the second line of \eqref{AbelLemPf} is noted by Jacobi in his 1835 paper \cite{Jacobi2}, which begins ``Of the theorems which are given in the elements of algebra, there is scarcely anything more useful in the most diverse equations'' than this vanishing result.  In \cite{Jacobi2}, Jacobi proves a similar vanishing for two equations in two unknowns, a result now called the \emph{Euler-Jacobi Vanishing Theorem}.  See \cite[1.5.5]{CD} for details and applications.
\end{itemize}
\end{remark}

We can now prove Abel's formula \eqref{ATModernS} for $dv$.

\begin{theorem}
\label{AThmdv}
With the above setup, assume \eqref{Assume1}.  Then
\[
dv = \sum_{i=1}^\mu f(x_i,y_i)\,dx_i = \sum_{j=1}^\alpha \res_{\beta_j}\hskip-1pt(S(x)) - \text{\large$\varPi$}(S(x)),
\]
where $\beta_1,\dots,\beta_\alpha$ are the roots of $f_2(x)F_0(x)$.
\end{theorem}

\begin{proof}
The proof uses the Global Residue Theorem  \eqref{GRT}.  As noted in \eqref{SumR}, Lemma \ref{fxiyiSxi} implies that
\[
dv = \sum_{i=1}^\mu f(x_i,y_i)\,dx_i = -\sum_{i=1}^\mu \res_{x_i}(S(x)).
\]
In order to apply the Global Residue Theorem, we regard  the differential form $S(x)$ as a rational function, which we do as above by thinking of $d\ua = da_1, da_2, \dots$ as variables. Applying the  Global Residue Theorem to 
$S(x)\,dx$ gives
\[
\sum_{\beta \in \PP^1} \res_\beta( S(x)\,dx) = 0.
\]
The roots $\beta_1,\dots,\beta_\alpha$ of $f_2(x) F_0(x)$ are algebraic over $\K$ and hence distinct from the roots $x_1,\dots,x_\mu$ of $F(x)$, which are transcendental over $\K$ since $dx_i \ne 0$. By Lemma~\ref{AbelLem},  the only possible poles of $S(x)\,dx$ are $\beta_1,\dots,\beta_\alpha,x_1,\dots,x_\mu,\infty$ (some may fail to be poles since we do not assume that $R(x)$ is relatively prime to $f_2(x) F_0(x)F(x)$). Then
\begin{align*}
0  &= \sum_{j=1}^\alpha \res_{\beta_j}(S(x)\,dx) + \res_\infty(S(x)\,dx) + \sum_{i=1}^\mu \res_{x_i}(S(x)\,dx) \\
 &= \sum_{j=1}^\alpha \res_{\beta_j}(S(x))  - \text{\large$\varPi$}(S(x)) + \sum_{i=1}^\mu \res_{x_i}(S(x)),
\end{align*}
where the second line uses \eqref{PiRes}.  When we combine this with the above expression for $dv$, the theorem follows immediately.
\end{proof}

\subsection{Examples}
\label{AEx}

Here we give examples of how to use the formula
\begin{equation}
\label{61Sx}
\sum_{i=1}^\mu f(x_i,y_i)\,dx_i = \sum_{j=1}^\alpha \res_{\beta_j}(S(x))  - \text{\large$\varPi$}(S(x)) 
\end{equation}
from Theorem \ref{AThmdv}, where $\beta_1,\dots,\beta_\alpha$ are the roots of $f_2(x)F_0(x)$.

The first step in computing \eqref{61Sx} is to compute $S(x)$.  Recall that
\[
S(x) = \frac{R(x)}{f_2(x)F_0(x)F(x)}. 
\]
To find $R(x)$, we follow the proof of Lemma \ref{AbelLem} and write
\begin{equation}
\label{Sx1}
\Omega_1 =   f_1(x,y^{(1)})\, \prod_{k\ne 1}\theta(x,y^{(k)}) \,\delta \theta(x,y^{(1)}) =  R(x) \big(y^{(1)}\big)^{n-1} + R^{(1)}(x,y^{(1)}),
\end{equation}
where $R^{(1)}$ has degree $\le n-2$ in $y^{(1)}$.   (The proof of the lemma explains why we only need $\Omega_1$.)

The next step is to compute the roots $\beta_1,\dots,\beta_\alpha$ of $f_2(x)F_0(x)$ by symbolic methods.   From here, one could find the Laurent expansions of $S(x)$ at these roots and also at $\infty$, and then compute \eqref{61Sx}.

There is also an approach that avoids Laurent expansions.  Let $\nu_j$ be the  multiplicity of $\beta_j$ as a root of $f_2(x)F_0(x)$.  Since $R(x)$ is a polynomial differential form and $f_2(x)F_0(x)$ and $F(x)$ have no roots in common, $\nu_j$ is greater than or equal to the order of $\beta_j$ as pole of 
\[
S(x) =  \frac{R(x)}{f_2(x)F_0(x)F(x)}.
\]
Thus, by \eqref{computeres}, we have
\[
\res_{\beta_j}(S(x)) 
=  \frac1{(\nu_j-1)!} \frac{d^{\nu_j-1}}{dx^{\nu_j-1}} \big((x-\beta_j)^{\nu_j} S(x) \big)\Big|_{x=\beta_j}.
\]

We now present examples of \eqref{61Sx} where $\chi(x,y)$ and $\theta(x,y,\ua)$ are given by
\[
\chi(x,y) = y^2 - x^3 - 1 \ \text{ and } \ \theta(x,y,a,b,c) = x^2+ax+b+cy.
\]
Let the roots of $\chi(x,y)$ be $y^{(1)} = \y$ and $y^{(2)} = -\y$, where $\y^2 = x^3+1$.  Then 
\begin{align*}
r(x) &= \theta(y^{(1)})\theta(y^{(2)}) = \theta(\y)\theta(-\y)\\
&= (x^2+ax+b+c\y)(x^2+ax+b-c\y) = (x^2+ax+b)^2 - c^2(x^3+1).
\end{align*}
Note that $F_0(x) = 1$, so $r(x) = F(x)$.  The solutions of $\chi(x,y) = \theta(x,y) = 0$ are 
\[
(x_1,y_1),\ (x_2,y_2),\ (x_3,y_3),\ (x_4,y_4), 
\]
where $x_1,x_2, x_3, x_4$ are the roots of $r(x)$ and $y_i = -(x_i^2+ax_i+b)/c$ for $1 \le i \le 4$.  

For a rational function of the form
\[
f(x,y) = \frac{f_1(x,y)}{f_2(x) \chi'(y)} = \frac{f_1(x,y)}{f_2(x) \cdot 2y},
\]
we can compute $\sum_{i=1}^4 f(x_i,y_i)\,dx_i$ using \eqref{61Sx}. We illustrate how this works for two choices of $f(x,y)$.  

\begin{example}
\label{NewEx1a}
Consider
\[
f(x,y) = \frac{1}{xy} =  \frac{2}{x\cdot 2y} \ \text{ with } \ f_1(x,y) = 2,\  f_2(x) = x.
\]
Here, $f_2(x) F_0(x) = x$ and \eqref{61Sx} imply
\begin{equation}
\label{NE2}
\sum_{i=1}^4  \frac{dx_i}{x_i y_i} =  \res_0(S(x))-\text{\large$\varPi$}(S(x)).
\end{equation}
To find $S(x)$, we first compute the rational differential form $R(x)$.  By \eqref{Sx1}, 
\begin{equation}
\label{Omega1Ex}
\begin{aligned}
\Omega_1 &= f_1(x,y) \hskip1pt \theta(x,y^{(2)}) \hskip1pt \delta \theta(x,y^{(1)}) =  2 \hskip1pt \theta(x,-\y) \hskip1pt \delta \theta(x,\y) \\
&= 2 (x^2+ax + b - c\y)(x\,da + db + \y\,dc)\\
&= 2(x^2+ax + b )(x\,da + db)\\ &\quad + 2\big({-}c(x\,da + db) + (x^2+ax + b)\hskip1pt dc\big)\y - 2c \hskip1pt dc \hskip1pt \y^2.
\end{aligned}
\end{equation}
Since $\y^2 = x^3 + 1$ and $\y = y^{(1)}$, we obtain
\[
\Omega_1 =   \underbrace{2\big({-}c(x\,da + db) + (x^2+ax + b)\hskip1pt dc\big)}_{\displaystyle R(x)}y^{(1)} + \cdots.
\]
Thus
\[
R(x) = 2(-cx\,da - c\, db + (x^2+ax+b)\hskip1pt dc).
\]
Here, $S(x) = \frac{R(x)}{xr(x)}$.  Then $x = 0$ is a simple pole of $S(x)$, so by \eqref{computeres}, the residue is
\[
\res_0(S(x)) = \frac{R(0)}{r(0)} = \frac{2(-c\hskip1pt 0\,da - c\, db + (0^2+a\hskip.5pt 0+b)\hskip1pt dc)}{(0^2+0x+b)^2-c^2(0^3+1)} 
= \frac{-2c\,db + 2b\,dc}{b^2-c^2}.
\]
The $\text{\large$\varPi$}$ term is also easy since
\[
\text{\large$\varPi$}(S(x)) = \text{\large$\varPi$}\Big(\frac{2(-cx\,da - c\, db + (x^2+ax+b)\hskip1pt dc)}{x((x^2+ax+b)^2-c^2(x^3+1))}\Big) = \text{\large$\varPi$}\Big(\frac{2x^2 \, dc + \cdots}{x^5+ \cdots} \Big)= 0.
\]
It follows from \eqref{NE2} that 
\[
\sum_{i=1}^4 \frac{dx_i}{x_iy_i} =  \frac{-2c\,db + 2b\,dc}{b^2-c^2}.
\]
One can check that the differential form on the right is closed.  In fact, it is exact provided that logarithms are allowed:
\[
\frac{-2c\,db + 2b\,dc}{b^2-c^2} = \delta\log\Big(\frac{b+c}{b-c}\Big),
\]
so that
\[
\sum_{i=1}^4 \int \frac{dx_i}{x_iy_i} = \log\Big(\frac{b+c}{b-c}\Big) + C.
\]
We will see in Example \ref{NewEx3a} that when we compute $\sum_{i=1}^4 \int\!\frac{dx_i}{x_iy_i}$ using Abel's formula for $v$, the logarithm appears in a completely natural way.\hfill$\Diamond$
\end{example}

\begin{example}
\label{NewEx1b}
Next, consider
\[
f(x,y) = \frac{1}{x^2y} =  \frac{2}{x^2\cdot 2y} \ \text{ with } \ f_1(x,y) = 2,\  f_2(x) = x^2.
\]
Then $f_2(x) F_0(x) = x^2$ and \eqref{61Sx} imply that
\begin{equation}
\label{NE3}
\sum_{i=1}^4 \frac{dx_i}{x_i^2y_i} = \res_0(S(x))-\text{\large$\varPi$}(S(x)).
\end{equation}
Here, $R(x)$ is unchanged from Example \ref{NewEx1a}, and since $S(x) = \frac{R(x)}{x^2 r(x)}$, it is easy to see that the $\text{\large$\varPi$}$ term vanishes.  So we need only find the residue at $x = 0$. 
Since $x = 0$ is a double pole of $S(x)=\frac{R(x)}{x^2 r(x)}$, \eqref{computeres} tells us that the residue is
\[
\res_0\Big(\frac{R(x)}{x^2 r(x)} \Big) = \frac{d}{dx}\frac{R(x)}{r(x)}\Big|_{x=0} = \frac{R'(0)r(0) - R(0)r'(0)}{r(0)^2},
\]
and then \eqref{NE3} leads without difficulty to the formula
\[
\sum_{i=1}^4 \frac{dx_i}{x_i^2y_i} = \frac{-2c(b^2-c^2)\hskip1pt da + 4abc\,db - 2a(b^2+c^2)\hskip1pt dc}{(b^2-c^2)^2}.
\]
It is less obvious that 
\[
\sum_{i=1}^4 \int\!\frac{dx_i}{x_i^2y_i} =  -\frac{2ac}{b^2-c^2} + C,
\]
but this will become clear in Example \ref{NewEx3b}.\hfill$\Diamond$
\end{example}

\section{Abel's Formula for $v$}
\label{Abelv}

The goal of this section is to state and prove a modern version of Abel's formula for $v = \sum_{i=1}^\mu \int\! f(x_i,y_i)\, dx_i$.  Compared to our discussion of $dv$ in Section \ref{Abeldv}, this will take more thought because of the presence of logarithms.  Examples will be given in Section \ref{AbelExv}, and Section \ref{Limits} will discuss $v$ as a function on the Riemann surface of $\chi(x,y) = 0$.

\subsection{A Modern Version of Abel's Formula for \emph{v}}
\label{AvMV}

In one sense, a modern version of Abel's Formula for $v$ is easy to state.  Recall from \eqref{ATModern} that
\[
dv = \sum_{j=1}^\alpha \res_{\beta_j}\hskip-1pt \Big(\sum_{\ell=1}^n f(x,y^{(\ell)}) \hskip1pt\frac{\delta \theta(y^{(\ell)})}{\theta(y^{(\ell)})}\Big)  - \text{\large$\varPi$}\Big( \sum_{\ell=1}^n f(x,y^{(\ell)}) \hskip1pt\frac{\delta \theta(y^{(\ell)})}{\theta(y^{(\ell)})} \Big) ,
\]
where the $\beta_j$ are the roots of $f_2(x) F_0(x)$.  As usual, $\delta$ is the differential with respect to the indeterminates $\ua = a_1, a_2, \dots$.  Since the $a_i$ only appear in the $\theta$ terms and
\[
\frac{\delta \theta(y^{(\ell)})}{\theta(y^{(\ell)})} = \delta \log\theta(y^{(\ell)}),
\]
we see that, at least formally, the integrated version of the above formula for $dv$ is given by
\begin{equation}
\label{Line3}
v =   \sum_{j=1}^\alpha \res_{\beta_j}\hskip-1pt \Big(\sum_{\ell=1}^n f(x,y^{(\ell)}) \hskip1pt\log\theta(y^{(\ell)})\Big)
- \text{\large$\varPi$}\Big( \sum_{\ell=1}^n f(x,y^{(\ell)}) \hskip1pt \log\theta(y^{(\ell)})\Big)  + C,
\end{equation}
where $C$ is a constant of integration. This can be regarded as a modern version of Abel's original formula 
\[
v = C-\text{\large$\varPi$} \frac{F_2x}{\theta_1x} \sum \frac{f_1(x,y)}{\chi'y}  \log \theta y
+ \sum{}\rule{0pt}{12pt}'\nu \frac{d^{\nu-1}}{dx^{\nu-1}} \bigg\{\frac{F_2x}{\theta_1^{(\nu)}x} \sum \frac{f_1(x,y)}{\chi'y}  \log \theta y\bigg\}.
\]
from \eqref{Abel37}.  It is first main result of Abel's Paris memoir \cite[(37) on p.\ 159]{A1}.

Making all of this rigorous will take some work.  First, we need to give an algorithmic meaning to the expression
\begin{equation}
\label{Thetalog}
\sum_{j=1}^\alpha \res_{\beta_j}\hskip-1pt \Big(\sum_{\ell=1}^n f(x,y^{(\ell)}) \hskip1pt\log\theta(y^{(\ell)})\Big)
 - \text{\large$\varPi$}\Big( \sum_{\ell=1}^n f(x,y^{(\ell)}) \hskip1pt \log\theta(y^{(\ell)})\Big) 
 \end{equation}
 and second, with this algorithmic meaning in mind, we need to prove that
 \begin{equation}
 \label{dThetalog}
 dv = \delta\Big(\! \sum_{j=1}^\alpha \res_{\beta_j}\hskip-1pt \Big(\sum_{\ell=1}^n f(x,y^{(\ell)}) \hskip.5pt\log\theta(y^{(\ell)})\Big)-\text{\large$\varPi$}\Big( \sum_{\ell=1}^n f(x,y^{(\ell)}) \hskip.5pt \log\theta(y^{(\ell)})\Big)  \!\Big).
 \end{equation}
Since $dv = \sum_{i=1}^\mu f(x_i,y_i)\,dx_i$ is rational differential in the indeterminates $\ua$ and $\delta$ is the differential with respect to $\ua$, \eqref{dThetalog} makes it clear what \eqref{Line3} means.

\subsection{Puiseux Series and Logarithms}
\label{AvPW}

To make sense of \eqref{Thetalog}, we need residues and logs.  For residues, the series expansions are more complicated because they often involve fractional exponents.  Given $\beta \in \overline{\K}$ (the algebraic closure of $\K$), it is well known that the field 
\[
\bigcup_{h=1}^\infty \overline{\K}((x-\beta)^{1/h}))
\]
is algebraically closed (see, for example, \cite[Chapter IV, Theorem 3.1]{Walker}).  Elements of this field are series in $x-\beta$ with fractional exponents and coefficients in $\overline{\K}$, where the exponents have bounded denominators and only finitely many are negative.  These are commonly called \emph{Puiseux series} in $x-\beta$.  Puiseux series in $1/x$ are defined similarly.  

 It follows that $\chi(x,y)$ splits completely in this field, so that the roots $y^{(\ell)}$ can be written as Puiseux series in $x-\beta$ or $1/x$.  Algorithms for computing these Puiseux series are described by Walker \cite[Chapter IV, \S3]{Walker} and (in full detail) Edwards \cite[Essays 4.4 and 8.4--8.7]{Edwards}.  

The version of Abel's formula for $v$ given in Theorem \ref{justifyint} uses the logarithms of certain nonzero elements $c \in \overline{\K}[\ua]$.  Recall that $\overline{\K}[\ua]$ has commuting derivations $\frac{\partial}{\partial a_i}$.  By \cite[Section 3.2]{Bronstein}, there is a differential extension of $\overline{\K}(\ua)$ containing an element $L$ such that
\[
\frac{\partial L}{\partial a_i} = \frac1c \frac{\partial c}{\partial a_i}
\]
for all $i$.  The discussion following \cite[Definition 5.1.3]{Bronstein} allows us to write $L = \log(c)$, so that the above equation can be written
\begin{equation}
\label{logprop}
\frac{\partial \log(c)}{\partial a_i} = \frac1c \frac{\partial c}{\partial a_i}
\end{equation}
for all $i$.  Note that the derivations $\frac{\partial}{\partial a_i}$ continue to commute in this extension.

Turning to Abel's formula for $v$, the idea is to introduce only logs needed to compute the final answer.   A naive reading of \eqref{Thetalog} would be to start from
\begin{equation}
\label{naiveShat}
\widehat{S}(x) =  \sum_{\ell=1}^n f(x,y^{(\ell)})\log \theta(y^{(\ell)}).
\end{equation}
and  then interpret  \eqref{Thetalog}  to mean
\[
\sum_{j=1}^\alpha \res_{\beta_j}(\widehat{S}(x))  - \text{\large$\varPi$}(\widehat{S}(x)) ,
\]
where $\beta_1,\dots,\beta_\alpha$ are the roots of $f_2(x)F_0(x)$.   The challenge is that  $\log \theta(y^{(\ell)})$ can be poorly behaved at $\beta_j$ and $\infty$.  It may fail to have Puiseux expansions in the usual sense, and even when it does, the expansions will contain additional logs. 

Since we want to minimize the number of logs that appear, we will dispense with $\log \theta(y^{(\ell)})$.  Instead, for each $\beta_j$, we will replace it with a Puiseux series denoted $\log_{\beta_j}\! \theta(y^{(\ell)})$.   There will be a similar Puiseux series $\log_{\infty} \!\theta(y^{(\ell)})$ at $\infty$.

To define these series, take $\beta \in \overline{\K}$.  Since $y^{(\ell)}$ is integral over $x$ (remember that $\chi(x,y)$ is monic in $y$), its Puiseux expansion in $x - \beta$ has coefficients in $\overline{\K}$ and
no negative exponents.  For the polynomial $\theta(x,y,\ua)$, the indeterminates $\ua$ are some of the coefficients of $\theta(x,y,\ua)$ with respect to $x$ and $y$.  It follows that the Puiseux expansion of $\theta(y^{(\ell)}) = \theta(x,y^{(\ell)},\ua) $ in $x - \beta$ has no negative exponents and coefficients in 
\[
\overline{\K}[\ua]_{\le1} = \{\lambda_0 + \lambda_1 a_1 +  \lambda_2 a_2 + \cdots \mid \lambda_i \in \overline{\K}\}.
\]
Thus the Puiseux expansion is $\theta(y^{(\ell)}) =  c_{\beta,\ell}\,(x-\beta)^{e_{\beta,\ell}} + \cdots$, where $c_{\beta,\ell} \in \underline{\K}[\ua]_{\le1}$, $e_{\beta,\ell} \ge 0$ in $\Q$, and the omitted terms all involve higher powers (possibly fractional) of $x-\beta$.  This enables us to write 
\[
\theta(y^{(\ell)}) = c_{\beta,\ell}\hskip1.5pt (x-\beta)^{e_{\beta,\ell}}\hskip.5pt (1 + W_{\beta,\ell}),
\]
where $W_{\beta,\ell}$ is a Puiseux series involving only positive powers of  $x-\beta$.  In a similar way, for $\beta = \infty$, $\theta(y^{(\ell)})$ can be written
\[
\theta(y^{(\ell)}) = c_{\infty,\ell}\hskip1.5pt x^{e_{\infty,\ell}}\hskip.5pt (1 + W_{\infty,\ell}),
\]
where $W_{\infty,\ell}$ is now a Puiseux series involving only positive powers of  $\frac1x$. 

\begin{definition}
\label{logxi} 
For $\theta(y^{(\ell)}) = c_{\beta,\ell}\hskip1.5pt (x-\beta)^{e_{\beta,\ell}}\hskip.5pt (1 + W_{\beta,\ell})$ as above, we define
\[
\log_\beta \theta(y^{(\ell)}) = \log(c_{\beta,\ell}) + \sum_{k=1}^\infty \frac{(-1)^{k-1} W_{\beta,\ell}^k}{k}.
\]
Similarly, when $\theta(y^{(\ell)}) = c_{\infty,\ell}\hskip1.5pt x^{e_{\infty,\ell}}\hskip.5pt (1 + W_{\infty,\ell})$, we define
\[
\log_\infty \theta(y^{(\ell)})  = \log(c_{\infty,\ell}) + \sum_{k=1}^\infty \frac{(-1)^{k-1} W_{\infty,\ell}^k}{k}.
\]
\end{definition}

The logarithms $\log(c_{\beta,\ell})$ and $\log(c_{\infty,\ell})$ in Definition \ref{logxi} are from \eqref{logprop}.  Hence
\[
\delta\log(c_{\beta,\ell}) = \sum_i \frac{\partial \log(c_{\beta,\ell})}{\partial a_i} \hskip1pt da_i = \sum_i \frac1{c_{\beta,\ell}} \frac{\partial c_{\beta,\ell}}{\partial a_i} \hskip1pt da_i = \frac1{c_{\beta,\ell}} \sum_i \frac{\partial c_{\beta,\ell}}{\partial a_i} \hskip1pt da_i = \frac{\delta c_{\beta,\ell} }{ c_{\beta,\ell} },
\]
and similarly for $\log(c_{\infty,\ell})$.  Note that $\log_\beta \theta(y^{(\ell)})$ and $\log_\infty \theta(y^{(\ell)})$ are Puiseux series whose constant terms involve logs and whose other coefficients lie in $\overline{\K}(\ua)$.

The reader may wonder why Definition \ref{logxi} ignores $(x-\beta)^{e_{\beta,\ell}}$ and $x^{e_{\infty,\ell}}$. One reason is that in spite of missing these factors, $\log_\beta \theta(y^{(\ell)})$ and $\log_\infty \theta(y^{(\ell)})$ still behave correctly with respect to $\delta$, as shown by the following lemma. 

\begin{lemma}
\label{dlogxitheta} 
$\displaystyle{\delta\log_\beta \theta(y^{(\ell)}) = \frac{\delta\theta(y^{(\ell)})}{ \theta(y^{(\ell)})}}$ and \,$\displaystyle{\delta\log_\infty \theta(y^{(\ell)}) = \frac{\delta\theta(y^{(\ell)})}{ \theta(y^{(\ell)})}}$.
\end{lemma}

\begin{proof}  For simplicity, let $W = W_{\beta,\ell}$.  Then
\begin{align*}
\delta\log_\beta \theta(y^{(\ell)}) &= \delta\Big( \!\log(c_{\beta,\ell}) +\sum_{k=1}^\infty \frac{(-1)^{k-1} W^k}{k}\Big)
= \frac{\delta c_{\beta,\ell} }{ c_{\beta,\ell} } + \sum_{k=1}^\infty (-1)^{k-1} W^{k-1}\,\delta W\\
&= \frac{\delta c_{\beta,\ell} }{ c_{\beta,\ell} } + \frac{\delta W}{1{+}W} =  \frac{\delta c_{\beta,\ell} }{ c_{\beta,\ell} } + \frac{\delta(1{+}W)}{1{+}W} = \frac{(\delta c_{\beta,\ell})(1{+}W) + c_{\beta,\ell} \,\delta(1{+}W)}{c_{\beta,\ell} (1{+}W)}\\
&= \frac{\delta(c_{\beta,\ell} (1{+}W))}{c_{\beta,\ell} (1{+}W)} \cdot \frac{(x-\beta)^{e_{\beta,\ell}}}{(x-\beta)^{e_{\beta,\ell}}} \!=\!  \frac{\delta\big(c_{\beta,\ell} \,(x-\beta)^{e_{\beta,\ell}}(1{+}W)\big)}{c_{\beta,\ell}\,(x-\beta)^{e_{\beta,\ell}} (1{+}W)}  \!=\!  \frac{\delta\theta(y^{(\ell)})}{ \theta(y^{(\ell)})} ,
\end{align*}
where the second equality of the last line follows because $(x-\beta)^{e_{\beta,\ell}}$ is constant with respect to $\delta$.  The proof for $\delta\log_\infty \theta(y^{(\ell)})$ is similar and hence omitted.
\end{proof}

Another way to think about Definition \ref{logxi} is to naively assume the properties of logarithms and write
\begin{align*}
\log \theta(y^{(\ell)}) &= \log\big( c_{\beta,\ell}\hskip1.5pt (x-\beta)^{e_{\beta,\ell}}\hskip.5pt (1 + W_{\beta,\ell})\big)\\
&= \log( c_{\beta,\ell}) + e_{\beta,\ell} \log (x-\beta) + \log(1+W_{\beta,\ell})\\ 
&=  \log( c_{\beta,\ell}) + e_{\beta,\ell} \log (x-\beta) +  \sum_{k=1}^\infty \frac{(-1)^{k-1} W_{\beta,\ell}^k}{k},
\end{align*}
which agrees with $\log_\beta \theta(y^{(\ell)})$ except for the term $e_{\beta,\ell} \log (x-\beta)$.  This term is unwanted since it requires adding another log and unneeded since it is constant with respect to $\delta$.   

\subsection{Proof of Abel's Formula for \emph{v}}
\label{AvPf}

We begin with the algorithmic meaning of \eqref{Thetalog}.  First modify \eqref{naiveShat} by defining 
\[
\widehat{S}_\beta =  \sum_{\ell=1}^n f(x,y^{(\ell)})\log_\beta \theta(y^{(\ell)}),
\]
where $\log_\beta \theta(y^{(\ell)})$ is from Definition \ref{logxi} and $f(x,y^{(\ell)})$ is computed using the Puiseux series for $y^{(\ell)}$ at $x = \beta$.  Thus $\widehat{S}_\beta$ is a Puiseux series in $x-\beta$, where as many terms as needed can be computed algorithmically.  In a similar way, define
\[
\widehat{S}_\infty =  \sum_{\ell=1}^n f(x,y^{(\ell)})\log_\infty \theta(y^{(\ell)})
\]
as a Puiseux series in $\frac1x$.  Then we interpret  \eqref{Thetalog}  to mean
\[
 \sum_{j=1}^\alpha \res_{\beta_j}(\widehat{S}_{\beta_j}(x))-\text{\large$\varPi$}(\widehat{S}_\infty(x)).
\]
Only finitely many terms of the Puiseux series are needed to compute $\res_{\beta_j}(\widehat{S}_{\beta_j}(x))$ and $\text{\large$\varPi$}(\widehat{S}_\infty(x))$.  Thus  \eqref{Thetalog} now has a rigorous algorithmic meaning.   It follows that Abel's formula for $v = \sum_{i=1}^\mu \int\! f(x_i,y_i)\,dx_i$ will be proved once we prove \eqref{dThetalog}, which we rewrite as
\begin{equation}
\label{dThetalog2}
dv =  \delta\Big(\sum_{j=1}^\alpha \res_{\beta_j}(\widehat{S}_{\beta_j}(x)) \hskip-1pt -\text{\large$\varPi$}(\widehat{S}_\infty(x))\Big).
\end{equation}
This leads to the following theorem.

\begin{theorem}
\label{justifyint}
Assuming \eqref{Assume1}, $v = \sum_{i=1}^\mu \int\! f(x_i,y_i)\,dx_i$ is given by
\begin{align*}
v &= \sum_{j=1}^\alpha \res_{\beta_j}\hskip-1pt \Big(\sum_{\ell=1}^n f(x,y^{(\ell)}) \hskip1pt\log_{\beta_j}\theta(y^{(\ell)})\Big)
- \text{\large$\varPi$}\Big( \sum_{\ell=1}^n f(x,y^{(\ell)}) \hskip1pt \log_\infty\theta(y^{(\ell)})\Big) 
 + C\\
 &=  \sum_{j=1}^\alpha \res_{\beta_j}(\widehat{S}_{\beta_j}(x)) - \text{\large$\varPi$}(\widehat{S}_\infty(x)) +  C.
\end{align*}
\end{theorem}

\begin{proof}
As noted above, it suffices to prove \eqref{dThetalog2}.  By Theorem \ref{AThmdv}, we know that
\[
dv =\sum_{j=1}^\alpha \res_{\beta_j}(S(x)) -\text{\large$\varPi$}(S(x)).
\]
Comparing this to \eqref{dThetalog2}, we see that the theorem will follow once we prove that
\[
\delta\hskip1pt \res_{\beta_j}(\widehat{S}_{\beta_j}(x)) = \res_{\beta_j}(S(x)) \ \text{ and }\ \delta\hskip1pt \text{\large$\varPi$}(\widehat{S}_\infty(x)) = \text{\large$\varPi$}(S(x)).
\]
Since $\delta$ is the differential with respect to $\ua = a_1,a_2, \dots$, it operates on a Puiseux series in $x-\beta_j$ or $\frac1x$ by applying $\delta$ to the coefficients.  It follows that
\[
\delta\hskip1pt \res_{\beta_j}(\widehat{S}_{\beta_j}(x)) = \res_{\beta_j}(\delta \hskip1pt \widehat{S}_{\beta_j}(x)) \ \text{ and }\ \delta\hskip1pt \text{\large$\varPi$}(\widehat{S}_\infty(x)) =  \text{\large$\varPi$}(\delta\hskip1pt \widehat{S}_\infty(x)).
\]
We now have a straightforward computation:
\begin{align*}
\delta \hskip1pt \widehat{S}_{\beta_j}(x)) &= \delta\Big( \sum_{\ell=1}^n f(x,y^{(\ell)})\log_{\beta_j} \theta(y^{(\ell)}) \Big) = \sum_{\ell=1}^n f(x,y^{(\ell)}) \hskip1pt \delta\log_{\beta_j} \theta(y^{(\ell)})\\
&= \sum_{\ell=1}^n f(x,y^{(\ell)}) \frac{\delta\theta(y^{(\ell)})}{ \theta(y^{(\ell)})}= S(x).
\end{align*}
On the first line, the second equality follows since the indeterminates $\ua$ do not appear in $ f(x,y^{(\ell)})$.  On the second line, the first equality is Lemma \ref{dlogxitheta}, and the second equality is the definition of $S(x)$ given in \eqref{Sxdef}.  The equality $\delta\hskip1pt \text{\large$\varPi$}(\widehat{S}_\infty(x)) = \text{\large$\varPi$}(S(x))$ is proved similarly, so we are done.
\end{proof}

\begin{remark}
Theorem \ref{justifyint} gives an algorithm for computing $\sum_{i=1}^\mu \int\! f(x_i,y_i)\,dx_i$ and shows that the only logs that appear in the final answer (if any) are logs of linear combinations of 1 and $\ua = a_1,a_2,\dots$ with coefficients in $\overline{\K}$.  
\end{remark}

\subsection{Examples}
\label{AbelExv}

We return to Examples \ref{NewEx1a} and \ref{NewEx1b} and use Theorem \ref{justifyint} to get integral formulas.
As in Section \ref{AEx}, we have $\chi(x,y) = y^2 - x^3 - 1$ and $\theta(x,y,a,b,c) = x^2+ax+b+cy$.   Here, $y^{(1)} = \y$ and $y^{(2)} = -\y$ where $\y^2 = x^3+1$.   Also recall that $F_0(x) = 1$ and $F(x) = r(x) = (x^2+ax+b)^2 - c^2(x^3+1)$.  Our goal is to compute
\[
\sum_{i=1}^4\int\!  f(x_i,y_i)\,dx_i 
\]
for the two choices of $f(x,y)$ featured in  Examples \ref{NewEx1a} and \ref{NewEx1b}.

\begin{example}
\label{NewEx3a}
First consider $f(x,y) = \frac{1}{xy} =  \frac{2}{x\cdot 2y}$, so $f_2(x) = x$.  It follows that
\[
\sum_{i=1}^4 \int \frac{dx_i}{x_i y_i} = \res_0(\widehat{S}_0(x)) -\text{\large$\varPi$}(\widehat{S}_\infty(x)) + C.
\]

We begin with $\widehat{S}_0(x)$, which uses the Puiseux expansion
\[
\y = \sqrt{x^3+1} = 1 + \tfrac12 x^3 - \tfrac18 x^6 + \tfrac1{16}x^9 - \cdots.
\]
Thus
\begin{align*}
\theta(\y) &= x^2 + ax + b + c\big(1 + \tfrac12 x^3 - \tfrac18 x^6+ \cdots) = b+c + ax + x^2 +  \tfrac{c}2 x^3 - \cdots\\
&= (b+c)\Big(1 + \frac{a}{b+c} x + \frac{1}{b+c}x^2 + \cdots\Big).
\end{align*}
It follows that
\[
\frac{1}{x\y}\log_0 \theta(\y) = \frac{\log(b+c) + (\frac{a}{b+c} x + \frac{1}{b+c}x^2 + \cdots) - \cdots}{x(1 + \frac12x^3 + \cdots)} = \frac{\log(b+c)}x + \cdots.
\]
In a similar way, 
\[
\theta(-\y) = x^2 + ax + b - c\big(1 + \tfrac12 x^3 - \tfrac18 x^6+ \cdots) = b-c + ax + x^2 +  \tfrac{c}2 x^3 - \cdots
\]
leads to
\[
\frac{1}{x(-\y)}\log_0 \theta(-\y) = -\frac{\log(b-c)}x + \cdots.
\]
Adding these gives 
\begin{equation}
\label{S0xEx}
\widehat{S}_0(x) = \frac{1}{x\y}\log_0 \theta(\y) + \frac{1}{x(-\y)}\log_0 \theta(-\y) = \frac{\log(b+c) - \log(b-c)}x + \cdots.
\end{equation}


The next step is to compute $\widehat{S}_\infty(x)$.  
The Puiseux expansion of $\y$ in $\frac1x$ is easy to compute since $\y^2 = x^3+1$ allows us to assume that
\begin{align*}
\y &= \sqrt{x^3+1} = x^{3/2}\sqrt{1+\frac1{x^3}} =  x^{3/2}\Big(1 + \frac{1}{2x^3} - \frac1{8x^6} + \cdots\Big)\\
&= x^{3/2} + \frac{1}{2x^{3/2}} - \frac1{8x^{9/2}} - \cdots.
\end{align*}
 Then
\begin{align*}
\theta(\y) &= x^2 + a x + b + c\y = x^2 + a x + b + c\Big(x^{3/2} + \frac{1}{2x^{3/2}} - \frac1{8x^{9/2}} - \cdots\Big)\\
&= x^2 + c x^{3/2} + ax + b + \frac{c}{2x^{3/2}} - \cdots\\
&= x^2(1+W),\ W = \frac{c}{x^{1/2}} + \frac{a}{x} + \frac{b}{x^2} + \frac{c}{2x^{7/2}} - \cdots,
\end{align*}
so that by Definition \ref{logxi}, 
\begin{align*}
\log_\infty \theta(\y) &= W -\frac12 W^2 + \cdots = \Big(\frac{c}{x^{1/2}} + \frac{a}{x} +\cdots\Big) - \frac12 \Big(\frac{c}{x^{1/2}} + \frac{a}{x} +\cdots\Big) ^2 + \cdots\\
&= \frac{c}{x^{1/2}} + \frac{a-\frac12 c^2}{x} + \cdots.
\end{align*}
Hence
\[
\frac{1}{x\y} \log_\infty \theta(\y) = \frac{\frac{c}{x^{1/2}} + \frac{a-\frac12 c^2}{x} + \cdots}{x^{5/2}\big(1 + \frac{1}{2x^3} - \frac1{8x^6} - \cdots\big)} = \frac{c}{x^3} + \cdots.
\]
In a similar way, one computes that
\[
\frac{1}{x(-\y)} \log_\infty \theta(-\y) = \frac{c}{x^3} + \cdots,
\]
from which we conclude that 
\begin{equation}
\label{SinfxEx}
\widehat{S}_\infty(x) =  \frac{1}{x\y} \log_\infty \theta(\y) + \frac{1}{x(-\y)} \log_\infty \theta(-\y) = \frac{2c}{x^3} + \cdots.
\end{equation}

Using \eqref{S0xEx} and \eqref{SinfxEx}, we obtain
\begin{align*}
\sum_{i=1}^4 \int \frac{dx_i}{x_i y_i} &= \res_0(\widehat{S}_0(x)) -\text{\large$\varPi$}(\widehat{S}_\infty(x)) + C\\
&= \big(\hskip-1pt\log(b+c) - \log(b-c)\big) - 0 + C =  \log\Big(\frac{b+c}{b-c}\Big) + C,
\end{align*}
where we regard $\log\big(\frac{b+c}{b-c}\big)$ as shorthand for $\log(b+c) - \log(b-c)$ (remember that we are being conservative about introducing logs). 
This agrees with Example \ref{NewEx1a}, except that now we know exactly where the logarithm comes from.\hfill$\Diamond$
\end{example}

\begin{example}
\label{NewEx3b}
Finally, consider $f(x,y) = \frac{1}{x^2y} =  \frac{2}{x^2\cdot 2y}$, so $f_2(x) = x^2$.  Since the $\text{\large$\varPi$}$ term again vanishes, we have
\[
\sum_{i=1}^4 \int \frac{dx_i}{x_i^2 y_i} = \res_0(\widehat{S}_0(x))+ C,
\]
where
\[
\widehat{S}_0(x) =  \frac{1}{x^2\y}\log_0 \theta(\y)  + \frac{1}{x^2(-\y)}\log_0 \theta(-\y).
\]
Using the computations from Example \ref{NewEx3a}, we obtain
\begin{align*}
\frac{1}{x^2\y}\log_0 \theta(\y) &= \frac{\log(b+c) + (\frac{a}{b+c} x + \frac{1}{b+c}x^2 + \cdots) - \cdots}{x^2(1 + \frac12x^3 + \cdots)}\\ &= \frac{\log(b+c)}{x^2} + \frac{a}{b+c} \frac1x + \cdots,
\end{align*}
and similarly
\[
\frac{1}{x^2(-\y)}\log_0 \theta(-\y) = -\frac{\log(b-c)}{x^2} - \frac{a}{b-c} \frac1x + \cdots.
\]
Adding these gives $\widehat{S}_0(x)$.  Hence the residue at $x = 0$ is 
\[
\res_0(\widehat{S}_0(x)) = \frac{a}{b+c}  - \frac{a}{b-c}  = -\frac{2ac}{b^2-c^2},
\]
so that
\[
\sum_{i=1}^4 \int\! \frac{dx_i}{x_i^2y_i} =  -\frac{2ac}{b^2-c^2} + C.
\]
Since $\delta$ is the differential with respect to $a,b,c$, applying $\delta$ gives
\[
\sum_{i=1}^4  \frac{dx_i}{x_i^2y_i} = \frac{-2c}{b^2-c^2}\hskip1pt da + \frac{4abc}{(b^2-c^2)^2}\hskip1pt \,db -
\frac{2a(b^2+c^2)}{(b^2-c^2)^2}\hskip1pt dc,
\]
which agrees with what we found in Example \ref{NewEx1b}.\hfill$\Diamond$
\end{example}

\subsection{Indefinite Integrals and Limits of Integration}
\label{Limits} 
Let us say more about the integrals in the sum $v = \sum_{i=1}^\mu \int\! f(x_i,y_i)\,dx_i$.  

For the most part, Abel regards these as indefinite integrals, i.e., antiderivatives (see \cite[p.\ 145]{A1}).  We say ``most part" because Abel mentions limits of integration exactly once in \cite{A1},  where he says that $dv = \sum_{i=1}^\mu f(x_i,y_i)\,dx_i$ gives $v = \sum_{i=1}^\mu \int\! f(x_i,y_i)\,dx_i$ ``by integrating between certain limits" \cite[p.\ 149]{A1}.  Definite integrals appear in many of Abel's papers, though ones with a variable upper limit of integration  are rare.  An example occurs in his great paper on elliptic functions, where he writes \cite[p.\ 266]{A2}
\[
\alpha = \int_0^x \frac{dx}{\sqrt{(1-c^2x^2)(1+e^2x^2)}}.
\]
(This makes $\alpha$ a function of $x$; the inverse function studied by Abel in \cite{A2} is the first published appearance of an elliptic function in the modern sense.) See \cite[p.\ 403]{K1} for further comments about how Abel thought about Abelian integrals.

We can use limits of integration to give an analytic definition of $v$ as follows.  Consider the Riemann surface $S$ associated to the algebraic curve $\chi(x,y) = 0$.  One can show that the differential form $f(x,y) \, dx$ represents a rational differential form on $S$ (and all such forms arise in this way).  Fix a base point $O \in S$.  Given another point $P \in S$, pick a path $\Gamma$ from $O$ to $P$ that avoids the poles of $f(x,y) \, dx$.  Then $\int_\Gamma \hskip-.5pt f(x,y)\,dx$ is defined, and if we vary $P$ slightly, the integral becomes a holomorphic function of $P$.  Informally, we write the integral as $\int_O^P\! f(x,y)\,dx$ since changing $\Gamma$ to a different path from $O$ to the initial choice of $P$ changes the function by a constant.  

With this in mind, recall that the $(x_i,y_i)$ are solutions of $\chi(x,y) = \theta(x,y,\ua) = 0$.  They are smooth points of $\chi(x,y) = 0$ since $x_i$ is transcendental over $\K$ and hence give points on $S$. 
Using $\int_O^{(x_i,y_i)}\! f(x,y)\,dx$ instead of $\int\! f(x_i,y_i)\,dx_i$ gives the function
\begin{equation}
\label{vLimits}
v = \sum_{i=1}^\mu \int_O^{(x_i,y_i)}\! f(x,y)\,dx.
\end{equation}
But $x_i$ and $y_i$ are algebraic functions of the indeterminates $\ua = a_1, a_2, \dots$, so that if we vary the $a_i$ slightly (Abel expresses this by saying that the values of the $a_i$ ``are contained within certain limits'' \cite[p.\ 149]{A1}), the points  $(x_i,y_i)$ will also vary slightly, with the result that \eqref{vLimits} is a holomorphic function of the $a_i$.  This approach to $v$ is used by Brill and Noether \cite[p.\ 322]{BN} in 1894, Baker \cite[p.\ 210]{BakerBook} in 1897, and Forsyth \cite[p.\ 585]{F} in 1918 in their discussions of Abel's Theorem.  Also, in 1857, Boole wrote his version of \eqref{vLimits} as $v = \sum \int^{x_i}\! X\,dx$ \cite[p.\ 745]{B}.

The treatment by Forsyth \cite[pp.\ 579--586]{F} is especially useful since he proves
\[
v =   \sum_{j=1}^\alpha \res_{\beta_j}\hskip-1pt \Big(\sum_{\ell=1}^n f(x,y^{(\ell)}) \hskip1pt\log\theta(y^{(\ell)})\Big)
- \text{\large$\varPi$}\Big( \sum_{\ell=1}^n f(x,y^{(\ell)}) \hskip1pt \log\theta(y^{(\ell)})\Big)  + C,
\]
where $v$ is as above and the right-hand side uses the analytic interpretation of $\log\theta(y^{(\ell)})$ (he uses 
Boole's $\Theta$ operator from Definition \ref{Thetadef}). 
As noted by Forsyth, 
``To evaluate the right-hand side, only algebraic expansions are necessary; and the result will be some function of the
parameters'' \cite[p.\ 585]{F}.  These ``algebraic expansions" are (up to a constant) the Puiseux series used in Theorem \ref{justifyint}.  The result is that the version of Abel's formula  for $v$ stated in Theorem \ref{justifyint} gives an explicit formula for $v$ from \eqref{vLimits} as a rational and logarithmic function of the $a_i$.

We should also mention that the discussion of Abel's Theorem in Dieudonn\'e \cite{Dieudonne} and Kleiman \cite{K1} uses \eqref{vLimits}.  The interpretation of $v$ given in \eqref{vLimits} is needed in order to understand some of the deeper aspects of Abel's memoir.

\section[test]{Abel's Formula for $\gamma$}
\label{AbelHolo}
Abel's formula for $v = \sum_{i=1}^\mu \int\! f(x_i,y_i)\,dx_i$ states (in modern form) that
\[
v = \sum_{j=1}^\alpha \res_{\beta_j}\hskip-1pt \Big(\sum_{\ell=1}^n f(x,y^{(\ell)}) \log \theta(y^{(\ell)})\Big) -  \text{\large$\varPi$}\Big( \sum_{\ell=1}^n f(x,y^{(\ell)})\log \theta(y^{(\ell)})\Big) + C.
\]
He proves his version of the formula in Section~4 of \cite{A1}.  In the next section, he notes that in this equation, 
\begin{quote}
the second member [the right-hand side] is in general a function of the quantities $a, a', a''$, etc.  If one assumes that it is equal to a constant, this will generally result in certain relations among the quantities; but there are also certain cases for which the second member reduces to a constant, regardless of the values of the quantities $a, a', a''$, etc.  Let's seek out these cases:  \ \cite[p.\ 164]{A1}
\end{quote}

The cases Abel proposes to study lead to conditions on the differential $f(x,y)\,dx$, and 
the integer $\gamma$ appearing in \eqref{Abel62} counts the number of ``arbitrary constants'' that appear when $f(x,y)\,dx$ satisfies these conditions.  In the 19th century, the dimension of a vector space was described by the minimum number of constants needed to express a general element of the space.   It follows that $\gamma$ is the dimension of a vector space of differential forms that satisfy Abel's conditions.  

From a modern perspective, $v = \sum_{i=1}^\mu \int\! f(x_i,y_i)\,dx_i$ is guaranteed to be constant when $f(x,y)\, dx$ is a holomorphic differential on the Riemann surface $S$ associated to $\chi(x,y) = 0$.  Recall that $f(x,y)\, dx$ is \emph{holomorphic} on $S$ if it has no poles on $S$.  When $f(x,y)\,dx$ is holomorphic, one can prove that $dv =  \sum_{i=1}^\mu f(x_i,y_i)\,dx_i = 0$ by the methods of \cite[Essay 9.8]{Edwards}.  It follows that $v = \sum_{i=1}^\mu \int\! f(x_i,y_i)\,dx_i$ is constant.  A more elementary proof that $v$ is constant is sketched in \cite[p.\ 406]{K1}.  

The genus of the Riemann surface $S$ is the dimension of the vector space of holomorphic differentials on $S$.  In Theorem \ref{gammagenus}, we will prove that for most choices of $\chi(x,y)$, 
Abel's $\gamma$ is the genus of the Riemann surface $S$.   To prove Theorem~\ref{gammagenus}, we begin with Abel's conditions on $f(x,y)\,dx$ (Section~\ref{Abelgammasec}) and then follow Baker \cite{Baker} to interpret Abel's formula \eqref{Abel62} for $\gamma$ in terms of the Newton polygon of $\chi(x,y)$ (Section~\ref{NewtonPolygon}).  From here,  Theorem \ref{gammagenus} is an easy consequence of standard results about toric surfaces (Section~\ref{RelationGenus}).

The relation between $\gamma$ and the genus is our first hint of the real depth of Abel's memoir \cite{A1}.  We will glimpse more of the depth in Sections \ref{HoloDiff} when we discuss holomorphic differentials.  Although Abel never defines the genus in \cite{A1}, he keeps asking questions where the genus is the answer.  

\subsection{Abel's  Inequalities}
\label{Abelgammasec}
Requiring that $v$ be constant is equivalent to the equation
\[
0 = dv = \sum_{j=1}^\alpha \res_{\beta_j}\hskip-1pt \Big(\sum_{\ell=1}^n f(x,y^{(\ell)}) \hskip1pt\frac{\delta \theta(y^{(\ell)})}{\theta(y^{(\ell)})}\Big) -  \text{\large$\varPi$}\Big( \sum_{\ell=1}^n f(x,y^{(\ell)}) \hskip1pt\frac{\delta \theta(y^{(\ell)})}{\theta(y^{(\ell)})} \Big).
\]
The first sum is over the roots $\beta_j$ of $f_2(x) F_0(x)$, where $f(x,y) = \frac{f_1(x,y)}{f_2(x) \chi'(y)}$ and $r(x) = F_0(x) F(x,\ua)$ is the factorization of the resultant from \eqref{Efact}.  One way to guarantee the vanishing of the first sum  is to assume that $f_2(x)F_0(x)$ is constant.  Abel makes this assumption and then turns his attention to the equation
\begin{equation}
\label{Pizero}
0 =  \text{\large$\varPi$}\Big( \sum_{\ell=1}^n \frac{f_1(x,y^{(\ell)})}{\chi'(y^{(\ell)})} \hskip1pt\frac{\delta \theta(y^{(\ell)})}{\theta(y^{(\ell)})} \Big).
\end{equation}
Houzel  \cite[pp.\ 92--94]{Houzel} gives a nice description of Abel's analysis of \eqref{Pizero}.  Our approach will be slightly different.

Given a Puiseux series $R$ in $1/x$ (what the 19th century would call ``descending powers of $x$''), Abel \cite[p.\ 161]{A1} defines $hR$ to be the highest exponent (``haut exposant'') of $x$ that appears.  We denote this by $h(R)$, so that
\[
h(R) = h \text{ when } R = c x^h + \hskip.5pt \text{smaller powers of $x$ and $c \ne 0$}.
\]
Writing $R$ as a standard Puiseux series in $u = 1/x$ with increasing exponents, the usual valuation $\nu$ satisfies $\nu(R) = -h(R)$.  When $R$ is a polynomial in $x$, $h(R)$ is just its degree.

An easy calculation shows that 
\[
h\Big(\frac{\delta \theta(y^{(\ell)})}{\theta(y^{(\ell)})} \Big) \le 0, \quad \ell = 1,\dots,n.
\]
Since {\large$\varPi$} takes the coefficient of $\frac1x$, it follows easily that \eqref{Pizero} holds whenever 
\begin{equation}
\label{Abelminus1}
h\Big(\frac{f_1(x,y^{(\ell)})}{\chi'(y^{(\ell)})}\Big) < -1, \quad \ell = 1,\dots,n.
\end{equation}

Our study of \eqref{Abelminus1} begins with the simpler inequalities
\begin{equation}
\label{Abelminus1simple}
h\Big(\frac{x^a (y^{(\ell)})^m}{\chi'(y^{(\ell)})}\Big) < -1, \quad \ell = 1,\dots,n,
\end{equation}
where $a$ and $m$ are nonnegative integers, i.e., $(a,m) \in \Z^2_{\ge0}$.  Equivalently,
\[
a < h(\chi'(y^{(\ell)})) - m\hskip.5pt h(y^{(\ell)}) - 1, \quad \ell = 1,\dots,n.
\]
We will show in Theorem \ref{Abelgamma} below that Abel's formula for $\gamma$ in \eqref{Abel62} counts the number of pairs $(a,m)$ that satisfy \eqref{Abelminus1simple}.

Since $h(y^{(\ell)})$ appears frequently in what follows, we will write it more simply as $h_\ell = h(y^{(\ell)})$  (Abel uses $\varphi \ell$ instead of $h_\ell$ \cite[p.\ 175]{A1}).  Thus \eqref{Abelminus1simple} is equivalent to
\begin{equation}
\label{htmchi}
a < h(\chi'(y^{(\ell)})) - m\hskip.5pt h_\ell- 1, \quad \ell = 1,\dots,n.
\end{equation}

Following Abel \cite[p.\ 162]{A1}, we order the $y^{(\ell)}$ so that
\begin{equation}
\label{orderyell}
h_1 \ge h_2 \ge \cdots \ge h_n, \quad h_\ell = h(y^{(\ell)})
\end{equation}
and assume that the $y^{(\ell)}$ have distinct highest terms.  This implies that 
\begin{equation}
\label{hdiff}
h(y^{(\ell)}-y^{(k)}) = h(y^{(\ell)}) = h_\ell \ \text{whenever } \ell < k.
\end{equation}
Since $\chi(x,y) = \prod_{\ell=1}^n(y-y^{(\ell)})$ implies $\chi'(y^{(\ell)}) = \prod_{k \ne \ell} (y^{(\ell)}-y^{(k)})$, \eqref{hdiff} easily implies that \eqref{Abelminus1simple} and \eqref{htmchi} are equivalent to 
\begin{equation}
\label{htmineq}
a < h_1  + h_2  + \cdots + h_{\ell-1} + (n -\ell - m)h_\ell - 1, \quad \ell = 1,\dots,n.
\end{equation}
These inequalities have some surprises to offer, as we now explain.

 \subsection{The Newton Polygon and Abel's Number $\gamma$}
\label{NewtonPolygon}
To analyze \eqref{htmineq}, Abel writes $h_\ell = h(y^{(\ell)})$ as a fraction in lowest terms with positive denominator.  After some further work, he arrives at the complicated but explicit formula for $\gamma$ given in \eqref{Abel62}.   We will derive his formula in the proof of Theorem \ref{Abelgamma} given below.  Our approach will use a connection between Abel's $\gamma$ and the Newton polygon of $\chi(x,y)$ noticed by Baker \cite{Baker} in 1894.  

Recall that the \emph{Newton polygon} of $\chi(x,y)$ is the convex hull of the exponent vectors of the nonzero terms of $\chi(x,y)$.  We start with an example studied by Abel and Baker.

\begin{example}
\label{AbelBaker}
Section 8 of Abel's memoir \cite[pp.\ 181--185]{A1} is devoted to the case where
\[
\chi(x,y) = p_0 + p_1y + p_2y^2 + \cdots + p_{12}y^{12} + y^{13}
\]
and $p_0 = p_0(x),\dots,p_{12} = p_{12}(x)$ are polynomials of respective degrees 
\[
2,3,2,3,4,5,3,4,2,3,4,1,1.
\]
Following Baker \cite[Fig.\ 7 on p.\ 139]{Baker}, Figure \ref{NewtonPolygonFig} shows the Newton polygon of $\chi(x,y)$.
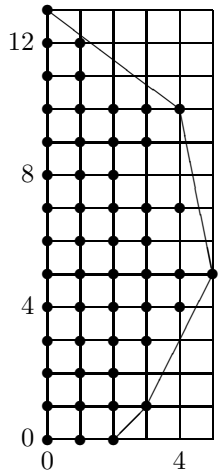
\begin{figure}[H]
\[
\begin{array}{c}
\setlength{\unitlength}{1.25pt}
\begin{picture}(70,135)
\multiput(10,10)(0,10){14}{\line(1,0){50}}
\multiput(10,10)(10,0){6}{\line(0,1){130}}
\multiput(10,10)(10,0){3}{\circle*{3}}
\multiput(10,20)(10,0){4}{\circle*{3}}
\multiput(10,30)(10,0){3}{\circle*{3}}
\multiput(10,40)(10,0){4}{\circle*{3}}
\multiput(10,50)(10,0){5}{\circle*{3}}
\multiput(10,60)(10,0){6}{\circle*{3}}
\multiput(10,70)(10,0){4}{\circle*{3}}
\multiput(10,80)(10,0){5}{\circle*{3}}
\multiput(10,90)(10,0){3}{\circle*{3}}
\multiput(10,100)(10,0){4}{\circle*{3}}
\multiput(10,110)(10,0){5}{\circle*{3}}
\multiput(10,120)(10,0){2}{\circle*{3}}
\multiput(10,130)(10,0){2}{\circle*{3}}
\put(10,140){\circle*{3}}
\put(10,140){\line(4,-3){40}}
\put(50,110){\line(1,-5){10}}
\put(40,20){\line(1,2){20}}
\put(30,10){\line(1,1){10}}
\put(8,1){$0$}
\put(48,1){$4$}
\put(2,8){$0$}
\put(2,48){$4$}
\put(2,88){$8$}
\put(-2,128){$12$}
\end{picture}
\end{array}
\]
\caption{Newton Polygon for Abel's Example}
\label{NewtonPolygonFig}
\end{figure}

Abel does not draw this picture, but he does compute the highest exponent $h_\ell = h(y^{(\ell)})$ of the solutions $y^{(1)},\dots,y^{(13)}$ of $\chi(x,y) = 0$.  His result \cite[p.\ 182]{A1} is
\[
h_1= h_2 = h_3 = \frac43\!  >  h_4 = \cdots = h_8  = \frac15 > h_9 =\cdots = h_{12} = \frac{-1}{\ 2}  > h_{13}=  -1 .
\]
Notice that these are the negative reciprocals of the slopes of the first quadrant edges in Figure \ref{NewtonPolygonFig}.
Abel uses the $h_\ell$ to compute the integers $n', m', \mu', n'', m'', \mu'', \dots$ appearing in his formula \eqref{Abel62} for $\gamma$, with the result that $\gamma = 38$ \cite[pp.\ 182--183]{A1}.

Methods for computing of Puiseux series in descending powers of $x$ have been known since Newton.  Unfortunately, most expositions of Puiseux series, including \cite{Edwards,Walker}, focus on series in ascending powers of $x-\beta$.  The only treatment for descending powers we know appears in the commentary to Minding's 1841 paper \cite{Minding}. This reference explains the link to the Newton polygon.

The above ordering of $h_1,\dots,h_{13}$ follows \eqref{orderyell}, which is needed for Abel's calculation of $\gamma$.  However, this ordering is also relevant to the Newton polygon in Figure \ref{NewtonPolygonFig}.  Ignoring the coordinate axes, the outer normals of the other edges of the Newton polygon are of the form $(1,h_\ell)$, which gives the outer normals
\[
\nu_1 = (1,\tfrac43),\ \nu_2 =  (1,\tfrac15),\ \nu_3 =  (1,\tfrac{-1}{\ 2}),\ \nu_4 =  (1,-1).
\]
The convexity in Figure \ref{NewtonPolygonFig} forces the  $y$-coordinates of the $\nu_i$ to decrease.

Lemma \ref{NewtonLem} below will prove that $(a,m)$ satisfies Abel's inequalities \eqref{htmineq} if and only if $(a + 1, m+1)$ is an interior lattice point of the Newton polygon of $\chi(x,y)$.  Proposition \ref{NewtonCor} and Remark \ref{Abelgammadef} will then show that Abel's initial definition of $\gamma$ is the number of interior lattice points, which for Figure \ref{NewtonPolygonFig} easily gives $\gamma = 38$.  \hfill$\Diamond$
\end{example}

In general, $\chi(x,y)$ is monic of degree $n$ in $y$ with Newton polygon $P$.  We will assume that $\chi(x,y)$ has a nonzero constant term, so that the line segment from the origin to $(0,n)$ is an edge of $P$, as in Figure \ref{NewtonPolygonFig}.  Aside from the coordinate axes, the remaining edges of $P$ are made up of $n$ line segments, where as we go from top to bottom, the endpoints of the line segments are
\[
Q_0 = (0,n),  Q_1 = (h_1,n-1), Q_2 = (h_1+h_2,n-2), \dots, Q_n = (h_1+ \cdots +h_n, 0),
\]
where $h_\ell = h(y^{(\ell)})$.  This follows from the commentary to \cite{Minding}.
The $\ell$th segment $Q_{\ell - 1}Q_\ell$ satisfies $Q_{\ell -1} - Q_\ell = (-h_\ell, 1)$ with outer normal $(1,h_\ell)$
relative to $P$ (so each edge of $P$ consists of one or more of these line segments). 
By \eqref{orderyell}, the outer normals $(1,h_1),\dots,(1,h_n)$ have decreasing $y$-coordinates as we go from top to bottom.  This is consistent with the convexity of $P$.  

In Figure \ref{NewtonPolygonFig}, the line segments $Q_0Q_1, Q_1Q_2, Q_2Q_3$ have outer normal $(1,\frac43)$ and form the top edge.  Similarly, the next five line segments have outer normal $(1,\frac15)$ and form the second edge, and so on.

The above description of $P$ implies that a point $Q = (u,v)$ in the plane lies in $P$ if and only if
\begin{equation}
\label{Pineq}
u \ge 0, \ v \ge 0, \text{ and } (Q - Q_{\ell-1}) \cdot (1,h_\ell) \le 0 \text{ for } \ell = 1,\dots,n.
\end{equation}
To explain the final inequalities, note that $Q$ lies on the line connecting $Q_{\ell -1}$ and $Q_\ell$ if and only 
$ (Q - Q_{\ell-1}) \cdot (1,h_\ell) = 0$.  Since $(1,h_\ell)$ is an outer normal, we must have $(Q - Q_{\ell-1}) \cdot (1,h_\ell)\le 0$ for points in $P$.

\begin{lemma}
\label{NewtonLem}
Let $P$ be the Newton polygon of $\chi(x,y)$, and assume that $\chi(x,y)$ has a nonzero constant term. Also, for the solutions $y^{(1)},\dots,y^{(n)}$ of $\chi(x,y) = 0$, assume that the highest terms of their Puiseux expansions in $1/x$ are distinct.  Then a lattice point $(a,m) \in \Z^2_{\ge0}$  satisfies the inequalities  \eqref{htmineq}, or equivalently
\eqref{Abelminus1simple},   if and only if $(a +1, m+1)$ is an interior lattice point of $P$.  
\end{lemma}

\begin{proof}
First, the hypothesis on the highest terms of the $y^{(\ell)}$ guarantees that \eqref{Abelminus1simple} and  \eqref{htmineq} are equivalent.  Then observe that $(a +1, m+1)$ lies in the interior of the first quadrant.  Thus it is an interior point of $P$ if and only if
\[
\big((a +1, m+1) - Q_{\ell-1}\big) \cdot (1,h_\ell) < 0, \quad \ell = 1,\dots, n.
\]
Using $Q_{\ell-1}  = (h_1+ \cdots + h_{\ell-1},n-\ell+1)$, these inequalities can be written as
\[
a +1 + (m+1)h_\ell - \big(h_1+ \cdots + h_{\ell-1} + (n-\ell+1)h_l\big) < 0, \quad \ell = 1,\dots, n,
\]
which are easily seen to be equivalent to \eqref{htmineq}.
\end{proof}

\begin{proposition}
\label{NewtonCor}
With the same hypotheses as Lemma \ref{NewtonLem}, consider the vector space $\Omega$ of differential forms $f(x,y)\hskip.5pt dx = \frac{f_1(x,y)}{\chi'(y)}\hskip.5pt dx$, where 
\[
f_1(x,y) = \sum c_{a,m} x^a y^m,\quad c_{a,m} \in \C
\]
and the sum is over all pairs $(a,m)$ for which $(a+1,m+1)$ is an interior lattice point of $P$.  Also, for $m = 0,\dots,n-2$, let $\tau_m$ be the largest integer such that $(\tau_m+1,m+1)$ is an interior lattice point of $P$, with $\tau_m = -1$ when $P$ has no such interior lattice point. Then{\rm:}
\begin{itemize}
\item[\rm(1)] $v = \sum_{i=1}^\mu \int\! f(x_i,y_i)\,dx_i$ is constant for every $f(x,y)\,dx$ in $\Omega$.
\item[\rm(2)] $\tau_m = \lceil h_1+\cdots + h_{n-m-1}\rceil - 2$.
\item[\rm(3)] $\Omega$ is a vector space of dimension
\begin{equation}
\label{mygammadef}
\dim(\Omega) = \#\{\text{\rm interior lattice points of } P\} = \sum_{m=0}^{n-2} (\tau_m + 1).
\end{equation}
\end{itemize}
\end{proposition}

\begin{proof}
For (1), let $x^a y^m$ be a monomial appearing in $f_1(x,y)$.  Then $(a,m)$ satisfies \eqref{Abelminus1simple} by Lemma \ref{NewtonLem}, which easily implies that $f_1(x,y)$ satisfies \eqref{Abelminus1}.  By the discussion preceding \eqref{Abelminus1}, it follows that $f_ 1(x,y)$ satisfies \eqref{Pizero}.  As explained at the beginning of Section \ref{Abelgammasec}, we conclude that $v = \sum_{i=1}^\mu \int\! f(x_i,y_i)\,dx_i$ is constant.

For (2), the discussion leading up to \eqref{Pineq} shows that $Q_\ell = (h_1+\cdots+h_\ell, n-\ell)$ is on the boundary of $P$ for $\ell = 0,\dots,n$.  Thus $Q_{n-m-1} = (h_1+\cdots+h_{n-m-1}, m+1)$ is on the right-hand side of the boundary of $P$.  If $P$ has a right-most interior lattice point at height $m+1$, then it can be written $(\tau_m+1,m+1)$, where
\[
\tau_m+1 = \lceil h_1+\cdots + h_{n-m-1}\rceil - 1, \text{ so } \hskip1.5pt \tau_m = \lceil h_1+\cdots + h_{n-m-1}\rceil - 2.
\]
If $P$ has no interior lattice points at height $m+1$, then $0 < h_1+\cdots + h_{n-m-1} \le 1$.  In this case, the above formula gives $\tau_m = -1$, which is what we want.

Finally, for (3), $\dim(\Omega)$ is clearly the number of interior lattice points of $P$.  Every interior lattice point can be written $(a+1,m+1)$ for integers $a,m \ge 0$, and note that $m + 1 < n$ since $(0,n)$ is the only point of $P$ at height $n$.  Thus $0 \le m \le n-2$.  Then $0 \le a \le \tau_m$ by the definition of $\tau_m$, so that the total number of interior lattice points is $\sum_{m=0}^{n-2} (\tau_m + 1)$.  Note that this works when $\tau_m = -1$.
\end{proof}

\begin{remark}
\label{Abelgammadef}
The sum $\sum_{m=0}^{n-2} (\tau_m + 1)$ is closely related to the definition of $\gamma$ given in \cite{A1}.  To analyze 
\eqref{Abelminus1}, Abel writes $f_1(x,y)$ in the form
\[
f_1(x,y) = t_0 + t_1 y + \cdots + t_{n-1} y^{n-1}, 
\]
where $t_m = t_m(x)$ is a polynomial in $x$ of degree $ht_m$ \cite[p.\ 162]{A1}.  On p.\ 163,  he takes ``la plus grande valeur de $ht_m$'', which is his version of our $\tau_m$ from Proposition~\ref{NewtonCor}.  For this value of $ht_m$, he designates  by $\gamma$ the number
\[
ht_1+ \cdots + ht_{n-2} + n - 1 = \sum_{m=0}^{n-2} (ht_m+1)
\]
(see \cite[p.\ 167]{A1}).
This is consistent with the sum $\sum_{m=0}^{n-2} (\tau_m + 1)$ in \eqref{mygammadef}.  For Abel, $\gamma$ is the number of ``arbitrary constants'' in $f_1(x,y)$ and hence in $\frac{f_1(x,y)}{\chi'(y)}\hskip.5pt dx$.

The formula for $\tau_m$ given in Proposition \ref{NewtonCor} compares nicely with Abel's formula for the biggest value of $ht_m$ \cite[(53) on p.\ 163]{A1}, namely
\[
ht_m = hy' + hy'' + \cdots + hy^{(n-m-1)} - 2 + \varepsilon_{n-m-1},
\]
where $hy^{(\ell)} = h_\ell$ and $\varepsilon_{n-m-1}$ is ``a positive number less than unity which makes this equation possible''  (as noted by Houzel \cite[p.\ 92]{Houzel}, $\varepsilon_{n-m-1}$ can be zero). 
\end{remark}

We now turn our attention to Abel's formula \eqref{Abel62} for $\gamma$, where we use \eqref{mygammadef} from Proposition \ref{NewtonCor} as the definition of $\gamma$.   Our main tool will be Pick's formula, which states that
\begin{equation}
\label{Pick}
\#\{\text{\rm interior lattice points of } P\} = \mathrm{Area}(P) -\tfrac12 \#\{\text{lattice points of } \partial P\} + 1,
\end{equation}
where $\partial P$ denotes the boundary of $P$.
A proof of \eqref{Pick} using toric geometry is given in \cite[Example 9.4.4]{CLS}.  See  \cite{GS} for references to more elementary proofs.

As above, we assume that $\chi(x,y)$ has a nonzero constant term.  The first step in using \eqref{Pick} is to compute the area of $P$. Number the first quadrant edges of $P$ from $1$ to $\varepsilon$, where the first edge starts at $(0,n)$ and the $\varepsilon$th edge ends on the $x$-axis.   Write the outer normal of the $i$th edge as
\[
\nu_i = \Big(1,\frac{m_i}{\mu_i}\Big),
\]
where $m_i, \mu_i \in \Z$, $\mu_i > 0$, and $\gcd(m_i, \mu_i) = 1$.  Also let $n_i+1$ be the number of lattice points on the $i$th edge.  Suppose the top endpoint of the $i$th edge is the lattice point $(a,b)$.  The edge has slope $\frac{-\mu_i}{\ m_i}$ with $\gcd(m_i, \mu_i) = 1$, so the next lattice point on the edge must be $(a+m_i,b-\mu_i)$.  Since the edge has $n_i+1$ lattice points,  the bottom endpoint of the $i$th edge is $(a+n_im_i, b-n_i\mu_i)$.  Thus the $i$th edge is made up of $n_i\mu_i$ of the line segments $Q_{\ell-1}Q_\ell$ defined above.  Figure \ref{2slices} shows the first two edges of $P$ when $m_2 \ge 0$.

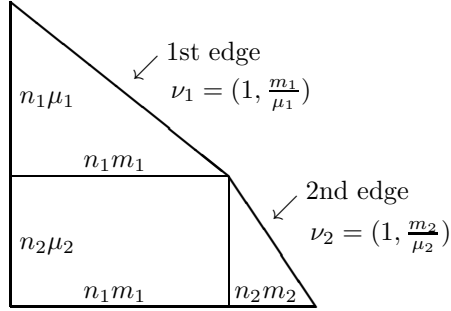
\begin{figure}[H]
\setlength{\unitlength}{1.1pt}
\[
\begin{array}{c}
\begin{picture}(155,105)
\thicklines
\put(0,0){\line(1,0){105}}
\put(0,0){\line(0,1){105}}
\put(0,105){\line(5,-4){75}}
\put(75,45){\line(2,-3){30}}
\thinlines
\put(0,45){\line(1,0){75}}
\put(75,0){\line(0,1){45}}
\put(25,3){$n_1m_1$}
\put(25,48){$n_1m_1$}
\put(77,3){$n_2m_2$}
\put(3,20){$n_2\mu_2$}
\put(3,70){$n_1\mu_1$}
\put(90,32){$\swarrow$}
\put(101.5,37){2nd edge}
\put(103,23){$\nu_2 = (1,\tfrac{m_2}{\mu_2})$}
\put(42,80){$\swarrow$}
\put(53.5,85){1st edge}
\put(55,72){$\nu_1 = (1,\tfrac{m_1}{\mu_1})$}
\end{picture}
\end{array}
\]
\caption{First Two Edges of the Newton Polygon}
\label{2slices}
\end{figure}

\begin{lemma}
\label{AreaP}
When the constant term of $\chi(x,y)$ is nonzero, the Newton polygon $P$ of $\chi(x,y)$  has area
\begin{align*}
\mathrm{Area}(P) =\ &n_1\mu_1 \cdot \frac{n_1 m_1}2 + n_2\mu_2\Big( n_1 m_1 +  \frac{n_2 m_2}2\Big) + 
n_3\mu_3\Big( n_1 m_1 + n_2 m_2 +  \frac{n_3 m_3}2\Big)\\ &+ \cdots  +
n_\varepsilon\mu_\varepsilon\Big( n_1 m_1 + \cdots + n_{\varepsilon-1} m_{\varepsilon - 1} +  \frac{n_\varepsilon m_\varepsilon}2\Big).
\end{align*}
\end{lemma}

\begin{proof}
Slice $P$ horizontally at the endpoints of the first quadrant edges.  This gives $\varepsilon$ slices, where the top slice is a triangle and the remaining slices are trapezoids (except possibly for the bottom slice, which might be a triangle).  The first two slices are shown in Figure \ref{2slices}.  The area of the triangular first slice is
\[
\frac12 \cdot n_1\mu_1 \cdot n_1 m_1 = n_1\mu_1 \cdot \frac{n_1 m_1}2,
\]
and for the trapezoidal second slice, the area is
\[
n_2\mu_2 \cdot \frac{ n_1 m_1 + ( n_1 m_1 +  n_2 m_2)}2 =  n_2\mu_2\Big( n_1 m_1 +  \frac{n_2 m_2}2\Big).
\]
The areas of the remaining slices are computed similarly, and the lemma follows.
\end{proof}

The next step is to compute the number of lattice points on the boundary of $P$.  The left edge of $P$ goes from the origin to $(0,n)$ and has $n+1$ lattice points, and the first quadrant edges have $n_1+1,n_2+1, \dots, n_\varepsilon + 1$ lattice points.  The bottom edge goes from the origin to $(n_1m_1 + \cdots + n_\varepsilon m_\varepsilon, 0)$ (this follows by extending Figure~\ref{2slices} to show all edges).  Hence the bottom edge has $n_1m_1 + \cdots + n_\varepsilon m_\varepsilon + 1$ lattice points.  (When $n_1m_1 + \cdots + n_\varepsilon m_\varepsilon = 0$, the bottom edge reduces to the origin, but $n_1m_1 + \cdots + n_\varepsilon m_\varepsilon + 1 = 1$ still counts the  number of lattice points.) Putting this all together and being careful not to double-count the vertices of $P$, we get
\[
n + n_1  + \cdots + n_\varepsilon + n_1m_1 + \cdots + n_\varepsilon m_\varepsilon 
\]
lattice points on the boundary of $P$.  Using $n = n_1 \mu_1 + \cdots + n_\varepsilon \mu_\varepsilon$ (the heights of the $\varepsilon$ slices add up to the total height $n$), we can rewrite this as 
\begin{equation}
\label{boundary}
n_1 \mu_1 + \cdots + n_\varepsilon \mu_\varepsilon + n_1(m_1+1) + \cdots + n_\varepsilon (m_\varepsilon +1).
\end{equation}
We are now ready to state and prove Abel's formula for $\gamma$ \cite[(62) on p.\ 168]{A1}.

\begin{theorem}
\label{Abelgamma}
When the constant term of $\chi(x,y)$ is nonzero, the number $\gamma$ of interior lattice points of the Newton polygon $P$ of $\chi(x,y)$ is given by
\begin{align*}
\gamma = \ &n_1\mu_1 \Big( \frac{n_1 m_1 -1}2\Big) + n_2\mu_2\Big( n_1 m_1 +  \frac{n_2 m_2-1}2\Big)\ + \ \\ 
&n_3\mu_3\Big( n_1 m_1 + n_2 m_2 +  \frac{n_3 m_3-1}2\Big) + \cdots  + \ \\
&n_\varepsilon\mu_\varepsilon\Big( n_1 m_1 + \cdots + n_{\varepsilon-1} m_{\varepsilon - 1} +  \frac{n_\varepsilon m_\varepsilon-1}2\Big)\\
&-\ \frac{n_1(m_1+1)}2 - \cdots - \frac{n_\varepsilon (m_\varepsilon +1)}2 + 1.
\end{align*}
\end{theorem}

\begin{proof}
By Pick's formula \eqref{Pick}, the number of interior lattice points of $P$ is
\[
\mathrm{Area}(P) - \tfrac12 \#\{\text{boundary lattice points of } P\} + 1.
\]
We computed the area of $P$ in Lemma \ref{AreaP} and the number of boundary points in \eqref{boundary}.  Substituting these results into Pick's formula and doing some easy algebra gives the formula for $\gamma$ stated in the theorem.
\end{proof}

\begin{remark}
With some slight adjustments in notation (replace $\mu_1, \mu_2, \dots, \mu_\varepsilon$ with $\mu', \mu'', \dots,  \mu^{(\varepsilon)}$, and similarly for $n_i$ and $m_i$),  the formula  in Theorem \ref{Abelgamma} is identical to \eqref{Abel62}, which is Abel's formula (62) on \cite[p.\ 168]{A1}.    

Abel introduces the fractions $\frac{m_1}{\mu_1},\dots,\frac{m_\varepsilon}{\mu_\varepsilon}$ in equation (54) on \cite[p.\ 164]{A1} as the distinct values that occur among $h_1,\dots,h_n$, just as we did above.  However, his definition of $n_i$ differs from ours.  Given one $\ell$ with $h_\ell = h(y^{(\ell)}) = \frac{m_i}{\mu_i}$, the Puiseux algorithm creates $\mu_i$ values of $\ell$ for which $h(y^{(\ell)}) = \frac{m_i}{\mu_i}$ (these $\mu_i$ solutions have the same highest exponent but their highest terms differ by $\mu_i$th roots of unity).  It follows that the number of $\ell$'s for which $h_\ell = h(y^{(\ell)}) = \frac{m_i}{\mu_i}$ is a multiple of $\mu_i$.  This is Abel's $n_i$, which he defines in equation (55) on \cite[p.\ 164]{A1}.  The Puiseux algorithm shows that Abel's $n_i$ agrees with ours.  

Our proof of Abel's formula for $\gamma$ is very different from Abel's original argument \cite[pp.\ 164--168]{A1}, though his proof, summarized by Houzel in \cite[pp.\ 92--93]{Houzel}, is complete modulo some minor details. 
\end{remark}

\subsection{Relation to the Genus}
\label{RelationGenus}
It remains to discuss when Abel's $\gamma$  is the genus in the modern sense.  This is a complicated question, especially because the curve $\chi(x,y) = 0$ may be singular, and further problems can arise at infinity when working in the projective plane. Sylow [4, Vol. 2, pp. 296--300], Brill and Noether [7, pp. 215--222], and Houzel [20, pp. 92--96] explore some of these challenges, and Kleiman discusses Abel's $\gamma$ and the genus in \cite[pp.\ 408--409]{K1}. 

We will take a different approach which focuses on the Newton polygon $P$ of $\chi(x,y) = 0$.  In toric geometry, the polygon $P$ determines a projective toric surface $X_P$ that contains a projective curve defined by $\chi(x,y) = 0$, similar to how this equation defines a curve in  the projective plane. The advantage of $X_P$ is that the curve defined by $\chi(x,y) = 0$ in $X_P$ often avoids the problems at infinity mentioned in the previous paragraph.  We refer to \cite{CLS} for an introduction to toric geometry.

Here is our main result about $\gamma$ and the genus.

\begin{theorem}
\label{gammagenus}
Assume that the constant term of $\chi(x,y)$ is nonzero.
If $\chi(x,y)$ is sufficiently general within its Newton polygon $P$, then the genus of the Riemann surface $S$ associated to $\chi(x,y) = 0$ is the integer $\gamma$ in Abel's formula \eqref{Abel62}, and $S$ can be identified with the smooth projective curve in $X_P$ defined by $\chi(x,y) = 0$.
\end{theorem}

\begin{proof}
The proof follows from results in \cite[\S10.5]{CLS}. As explained in the discussion preceding \cite[Proposition 10.5.8]{CLS}, standard Bertini Theorems imply that if $\chi(x,y)$ is sufficiently general within its Newton polygon $P$, then $\chi(x,y) = 0$ defines a smooth projective curve in $X_P$.  When this happens, the resulting smooth curve is the Riemann surface $S$ associated to $\chi(x,y) = 0$.

The same discussion in \cite{CLS} shows that when $\chi(x,y) = 0$ defines a smooth curve in $X_P$, its genus can be computed by the adjunction formula (see \cite[(10.5.1)]{CLS}).  Then the toric interpretation of the adjunction formula and Pick's formula \eqref{Pick} imply that the genus of $S$ is the number of interior lattice points of $P$ (see \cite[Prop.\ 10.5.8]{CLS} for the details).  By Theorem \ref{Abelgamma}, we conclude that the genus is given by Abel's formula for $\gamma$.
\end{proof}

\subsection{Holomorphic Differentials}
\label{HoloDiff}
Abel's $\gamma$ uses the inequalities \eqref{htmineq} that lead to the
vector space of differentials $\Omega$ in Proposition \ref{NewtonCor}.  On the other hand, the genus uses holomorphic differentials.  Theorem \ref{gammagenus} shows that $\gamma$ equals the genus in many cases but says nothing about how Abel's differentials relate to the holomorphic ones.  In fact, these differentials  are the same in the situation of the theorem. 

To see why, we use the classical theory of adjoint curves, as suggested by the reviewer and explained in \cite[pp.\ 407 and 410--412]{K1}. Recall that Abel assumes that $f_2(x)F_0(x)$ is constant, so that the differential form is
\begin{equation}
\label{AbelHoloform}
f(x,y)\,dx = \frac{f_1(x,y)}{\chi'(y)} \hskip.5pt dx.
\end{equation}
Proposition \ref{NewtonCor} can be interpreted as saying that \eqref{AbelHoloform} is one of Abel's differentials precisely when $xy f_1(x,y)$ has exponent vectors in the interior of the Newton polygon $P$ of $\chi(x,y)$.  When $\chi(x,y)$ is generic within $P$, Proposition I of Baker's paper \cite{Baker} then implies that Abel's differentials come from adjoint curves and hence are holomorphic by the theory of adjoint curves.
 
 Thus Abel's differentials are contained in the holomorphic ones.  These vector spaces have the same dimension by Theorem \ref{gammagenus}, so generically, Abel's differentials are precisely the holomorphic differentials on the Riemann surface $S$ of $\chi(x,y) = 0$.  

Once we connect holomorphic differentials to Abel's question about when $v = \sum_{i=1}^\mu \int\! f(x_i,y_i)\,dx_i$ is constant, other powerful ideas come into play, including linear equivalence of divisors, the Jacobian, and the Abel map (sometimes called the Abel-Jacobi map).  These topics are important but lie beyond the scope of these notes.  See  \cite[Sections 3 and 6]{K1} for further discussion.

We began this section with Abel's question about when his formula for $v$ reduces to a constant.  This led to some wonderful mathematics---far beyond anything Abel was thinking about in 1826---and gives another example of his amazing ability to ask the right question.   

\section{Other Versions of Abel's Theorem} 
\label{OtherAbelThms}
In this section we discuss the motivation for Abel's Paris memoir and say more about the other versions of Abel's Theorem that appear in this amazing paper.

\subsection{From Elliptic Integrals to Abel's Theorems}
Elliptic integrals have a long and rich history (see \cite{HouzelA} for a nice overview).  They enjoy many wonderful properties, including the classic addition theorem
\begin{equation}
\label{EulerAdd}
\int_0^x\!\! \frac1{\sqrt{1{-}t^4}}\,dt +\! \int_0^y \!\! \frac1{\sqrt{1{-}t^4}}\,dt  =\! \int_0^z \! \!\frac1{\sqrt{1{-}t^4}}\,dt, \ z = \frac{x\sqrt{1{-}y^4}+y\sqrt{1{-}x^4}}{1{+}x^2y^2}
\end{equation}
proved by Euler in 1753.  By Abel's time, it was known  that the sum of any number of elliptic integrals (which he called ``elliptic functions'') ``can be expressed by a single function of the same form, by adding a certain algebraic and logarithmic expression'' \cite[p.\ 444]{A3}.  When this happens, the variable in the ``single function" is an algebraic function of the variables in the elliptic summands, as in \eqref{EulerAdd}, where the variable $z$ of the ``single function'' on the right is related algebraically to the variables $x$ and $y$ of the integrals on the left.


In \eqref{EulerAdd}, also notice that the two integrals sum to a constant when we impose a single condition on $x$ and $y$ by requiring that $z$ be constant.
This is true for general elliptic integrals, provided we allow the sum to be an algebraic and logarithmic function.  In the introduction to his Paris memoir, Abel states this as follows:
\begin{quote}
we can express any sum of similar functions [elliptic integrals] by an algebraic and logarithmic function, provided that we establish a certain algebraic relationship between the variables of these functions. \cite[pp.\ 145--146]{A1}
\end{quote}
He goes on to explain his motivation for writing \cite{A1}:
\begin{quote}
This analogy between the properties of these functions [elliptic integrals] led the author to investigate whether or not it would be possible to find analogous properties of more general functions [the Abelian integrals \eqref{AbInt}] \cite[p.\ 146]{A1}
\end{quote}

Abel then states two versions of what we now call ``Abel's Theorem''.  For the first, he says: 
\begin{quote}
\raisebox{-4.5pt}{''}If one has several functions whose derivatives are roots of the same \raisebox{-4.5pt}{''}algebraic equation, all of whose coefficients are rational functions \raisebox{-4.5pt}{''}of the same variable, one can always express the sum of a number \raisebox{-4.5pt}{''}of similar functions by a \emph{logarithmic} and \emph{algebraic} function, pro- \raisebox{-4.5pt}{''}vided that a certain number of \emph{algebraic} relations are established \raisebox{-4.5pt}{''}between the variables of the functions in question.``{} \cite[p.~146]{A1}
\end{quote}
(The odd format using \raisebox{-4.5pt}{''} and `` is how Abel indicates a theorem in the introduction to \cite{A1}.)
The functions described at the beginning of this quote are Abelian integrals \eqref{AbInt}, so Abel is saying that the sum of any number of Abelian integrals is a logarithmic and algebraic function when a certain number of algebraic relations are satisfied.  He goes on to say:
\begin{quote}
The number of these relations does not depend on the number of summands, but only on the nature of the particular functions that one considers.  Thus, for example, for an elliptic integrand this number is 1; for an integrand that contains no irrationalities but a radical of the second degree, under which the variable has degree five or six, the number of necessary relations is 2, and so forth. 
\end{quote}
In the last sentence, Abel gives examples where the number of relations is easily seen to be the genus the curve in question.  This suggests that the genus may have a role to play in this version of Abel's Theorem, which Kleiman calls \emph{Abel's Relations Theorem} in \cite{K1}.  We will see below that the genus is indeed a key player.

Abel then states a second theorem:
\begin{quote}
\raisebox{-4.5pt}{''}One can always express the sum of a given number of functions, \raisebox{-4.5pt}{''}which are each multiplied by a rational number, and whose vari- \raisebox{-4.5pt}{''}ables are arbitrary, by a similar sum of a given number of func- \raisebox{-4.5pt}{''}tions, whose variables are algebraic functions of the variables of \raisebox{-4.5pt}{''}the given functions``{} \cite[p.~146]{A1}
\end{quote}
The ``given number" of functions in this quote is the number of relations in the previous theorem, which as we will see in  \eqref{mag} is the genus.  There is also a slight omission, for Abel forgot to say that the ``similar sum'' also involves a logarithmic and algebraic function.  This is included when he states the result later in the memoir \cite[(119) on p.\ 188]{A1}).  In \cite{K1}, Kleiman calls this \emph{Abel's Addition Theorem}. 

The key insight of  Abel's Addition Theorem is that a sum of Abelian integrals equals a logarithmic and algebraic function plus a given number, usually more than one, of Abelian integrals whose variables depend algebraically on the variables of the original integrals.  As just noted, the given number of integrals is the genus.

\subsection{Comments on Abel's Theorems}
The details of Abel's  Relations and Addition Theorems are described by Houzel \cite[Section 5]{Houzel}, and their implications for algebraic geometry in the 19th and 20th centuries are discussed by Kleiman \cite[Sections 4--7]{K1}.  Our goals here are more modest:\ we will (1) indicate how the sum 
\begin{equation}
\label{modestv}
v = \sum_{i=1}^\mu \int \! f(x_i,y_i)\,dx_i
\end{equation}
studied earlier in these notes leads to a special case of Abel's Relations Theorem and (2) say a few words about how the genus enters the picture.

In \eqref{modestv}, the $(x_i,y_i)$ are the solutions of $\chi(x,y) = \theta(x,y,\ua) = 0$ with $dx_i \ne 0$ and hence are algebraic functions of $\ua = a_1, a_2, \dots$.  Recall that the first result proved in \cite{A1} states that $v$ is a rational and logarithmic function of $\ua$. 

Following Abel \cite[p.\ 170]{A1}, let $\alpha$ denote the number of indeterminates $\ua = a_1,a_2,\dots,a_\alpha$ that appear in $\theta(x,y,\ua)$.  We will assume that $F_0(x) = 1$, which guarantees that $a_1,\dots,a_\alpha$ are algebraically independent.  Abel argues \cite[p.\ 170]{A1} that $x_1,\dots,x_\mu$ can be numbered so that $x_1,\dots,x_\alpha$ are algebraically independent.  This implies that $a_1,\dots,a_\alpha$, $x_{\alpha+1},\dots,x_\mu$, and $y_1,\dots,y_\mu$ are algebraic functions of $x_1,\dots,x_\alpha$.  Now write \eqref{modestv} in the form
\begin{equation}
\label{mualpha}
\sum_{i=1}^\alpha \int \! f(x_i,y_i)\,dx_i = v - \!\sum_{i=\alpha +1}^\mu \int \! f(x_i,y_i)\,dx_i.
\end{equation}
In this equation, $v$ is a rational and logarithmic function of $\ua$ and hence an algebraic and logarithmic function of $x_1,\dots,x_\alpha$.   There are $\mu-\alpha$ integrals on the right-hand side.  Abel assures us that this number is ``tr\'es-remarquable'' \cite[p.\ 172]{A1}.   

Abel's enthusiasm for $\mu-\alpha$ is due in part to its relation to the number $\gamma$ from \eqref{Abel62}. More precisely, when $F_0(x)$ is constant, he proves that
\[
\mu - \alpha \ge \gamma.
\]
He goes on to claim that equality can occur when $\theta(x,y,\ua)$ is chosen carefully (see \cite[pp.\ 176--177]{A1}).  
Several pages later, Abel \cite[(105) on p.\ 180]{A1} concludes that 
\begin{equation}
\label{mag}
\mu-\alpha = \gamma
\end{equation}
when all coefficients of the carefully chosen $\theta(x,y,\ua)$ are indeterminates (more on this below). By Theorem \ref{gammagenus}, $\mu-\alpha = \gamma$ is the genus in many cases.  

The full story is very complicated.  More details can be found in Sylow  \cite[Vol.\ 2, pp.\ 296--300]{A5}, Brill and Noether \cite[pp.\ 215--222]{BN}, and Houzel \cite[pp.\ 92--96]{Houzel}.  Kleiman's article \cite{K1} discusses these issues from the point of view of  both classical and modern algebraic geometry, while Edwards \cite[Chapter~9]{Edwards}, drawing on work of Dedekind and Weber \cite{DW}, takes a more constructive and algebraic point of view.  In particular, Essay 9.9 of \cite{Edwards} shows that $\mu-\alpha$ equals the genus for an explicitly described choice of $\theta(x,y,\ua)$.

For us, a link between $\mu - \alpha$ and the genus $g$ of the Riemann surface $S$ of $\chi(x,y) = 0$ can be described as follows (omitting many details).  Let $D$ be a divisor of degree $\mu \ge 2g-1$ on $S$.  As usual, $L(D)$ is the vector space of rational functions $\varphi(x,y)$ on $S$ such that $\mathrm{div}(\varphi(x,y))+D \ge 0$, and the corresponding complete linear system of divisors is the projective space $|D|$ of $L(D)$.  Set $\ell(D) = \dim L(D) = 1 + \alpha$, so that $\alpha = \dim |D|$.  Fix a basis $\varphi_0(x,y),\dots,\varphi_\alpha(x,y)$ of $L(D)$ and define
\[
\theta(x,y,\ua) = a_0\varphi_0(x,y)  + \cdots + a_\alpha \varphi_\alpha(x,y).
\]
For a generic choice of $\ua$, the solutions $(x_i,y_i)$ of $\chi(x,y) = \theta(x,y,\ua) = 0$ give an effective divisor of degree $\mu$ on $S$ in the complete linear system $|D|$ of dimension $\alpha$.   This is why the number of indeterminates in $\theta(x,y,\ua)$ is written $1+\alpha$, not $\alpha$.

By the Riemann-Roch Theorem (see, for example, \cite[\S8.6]{Fulton}), we have
\[
\deg(D) - g + 1 = \ell(D) - \ell(K-D),
\]
where $K$ is a canonical divisor on $S$.  But $\deg(D) = \mu$ and $\ell(D) = 1+ \alpha$, and $\mu \ge 2g-1$ guarantees that $\ell(K-D) = 0$.  Hence
\[
\mu - g + 1 = (1 + \alpha) - 0, \ \ \text{so that} \ \ \mu - \alpha = g.
\]
This argument gives one way of seeing why Abel needed to make a careful choice of $\theta(x,y,\ua)$.  It also justisfies his claim that $\mu-\alpha$ is ``very remarkable''. 

When $\mu-\alpha$ equals the genus, we get a version of Abel's Relations Theorem as follows.  In \eqref{mualpha}, impose $\mu-\alpha$ conditions by requiring that $x_{\alpha+1},\dots,x_\mu$ be constant.   Since $x_{\alpha+1},\dots,x_\mu$ are algebraic functions of  $x_1,\dots,x_\alpha$, this imposes $\mu-\alpha$ algebraic relations on $x_1,\dots,x_\alpha$, and when these conditions are satisfied, $dx_{\alpha+1} = \cdots = dx_\mu = 0$.  Then \eqref{mualpha} simplifies to
\[
\int \! f(x_1,y_1)\,dx_1 + \cdots + \int \! f(x_\alpha,y_\alpha)\,dx_\alpha = v,
\]
where $v$ is an algebraic and logarithmic function of $x_1,\dots,x_\alpha$.  The number of  conditions is $\mu-\alpha$, which is the genus in this situation.  Hence we have  a special case of Abel's Relations Theorem!

\appendix
\section{How Abel Proved His Formulas for \emph{dv} and \emph{v}}
\label{AbelPf}

The proofs of Abel's formulas for $dv$ and $v$ given in Sections \ref{Abeldv} and \ref{Abelv} differ from Abel's, mainly because he used partial fractions instead of residues.  To give the reader a sense of why partial fractions are relevant, Section~\ref{SecRational} will explain what Abel's formulas look like in the special case of rational functions.  Then Section~\ref{SecRoweAbel} will give his proofs of modern versions of the formulas for $dv$ and $v$, and Section~\ref{rewrite} will show how to transform Abel's original formula for $dv$ into the modern version.  Finally, Section~\ref{BooleRowe} will discuss how Boole and Rowe used residues to simplify Abel's proofs.

\subsection{The Case of Rational Functions}
\label{SecRational}

The integrals $\int \! f(x,y)\,dx$ in Abel's memoir can be complicated because $y$ is usually an algebraic function of $x$.  When $y$ is not present, we have the known case of an integral of a rational function.  Abel's formulas still apply and reveal an unexpected connection to partial fractions.  

We begin by applying  Abel's formula for $dv$ to a rational function.

\begin{theorem}
\label{Abel3A}
For a rational function $\varphi(x) \in \K(x)$ and a new variable $a$, we have
\begin{equation}
\label{Abel31A}
\varphi(a) =  \sum_{j=1}^\alpha \res_{x=\beta_j}\Big(\frac{\varphi(x)}{a-x}\Big) -\text{\large$\varPi$}\Big(\frac{\varphi(x)}{a-x}\Big),
\end{equation}
where $\beta_1,\dots,\beta_\alpha$ are the poles of $\varphi(x)$ in $\overline{\K}$.  
\end{theorem}

\begin{proof}
We apply Theorem \ref{AThmdv} to $\chi(x,y) = y-x$ and $\theta(x,y,a) = y-a$, so $\ua = a$.  The root of $\chi(x,y)$ is $y^{(1)} = x$, and the solution of $\chi(x,y) = \theta(x,y,a) = 0$ is $(x_1,y_1) = (a,a)$.  Thus 
\[
dv = \varphi(x_1)\, dx_1 = \varphi(a) \,da.  
\]
We also have $F(x) = x-a$ and $F_0(x) = 1$.   Following Abel \cite[(201) on p.\ 203]{A1}, set 
\[
\varphi(x) = \frac{f(x)}{f_2(x)},
\] 
so that the poles of $\varphi(x)$ are the roots of $f_2(x)$.  Note that $\chi'(y) = 1$ and $\delta = \frac{d}{da}\,da$.  

The definition of $S(x)$ given in \eqref{Sxdef} reduces to
\begin{equation}
\label{Sxrat}
S(x) = \varphi(x) \hskip1pt\frac{ \delta\theta(x,y^{(1)})}{\theta(x,y^{(1)})} = \varphi(x) \hskip1pt \frac{ \frac{d}{da}(x-a)\, da}{x-a} = \varphi(x)\frac{-\hskip1pt da}{x-a} = \frac{\varphi(x) \hskip1pt da}{a-x}.
\end{equation}
Theorem \ref{AThmdv} implies that if $\beta_1,\dots,\beta_\alpha$ are the roots of $f_2(x)$, then 
\[
 \varphi(a)\,da  =  \sum_{j=1}^\alpha \res_{\beta_j}\hskip-1pt(S(x)) - \text{\large$\varPi$}(S(x)) = \sum_{j=1}^\alpha \res_{\beta_j}\hskip-1pt\Big( \frac{\varphi(x) \hskip1pt da}{a-x}\Big) - \text{\large$\varPi$}\Big( \frac{\varphi(x) \hskip1pt da}{a-x}\Big). 
\]
Canceling $da$, we get \eqref{Abel31A}. 
\end{proof}

Now comes a surprise:\ \eqref{Abel31A} is an elegant way of writing the partial fraction decomposition of $\varphi(a)$.  Here is the precise result.

\begin{proposition}
\label{AbelPF1}
Let
\[
\varphi(x) =  \sum_{j=1}^\alpha \sum_{k = 1}^{\nu_j} \frac{A_j^{(k)}}{(x-\beta_j)^k} + G(x) 
\]
be the partial fraction decomposition of $\varphi(x)$, where the $A_j^{(k)}$ are constants, $G(x)$ is a polynomial, and $\nu_j$ is the order of $\beta_j$ as a pole of $\varphi(x)$.  Then
\begin{align}
\res_{x=\beta_j}\Big(\frac{\varphi(x)}{a-x}\Big) &= \sum_{k = 1}^{\nu_j} \frac{A_j^{(k)}}{(a-\beta_j)^k} \label{Abel32}\\
\text{\large$\varPi$}\Big(\frac{\varphi(x)}{a-x}\Big) &= -G(a). \label{Abel33}
\end{align}
Thus \eqref{Abel31A} computes the partial fraction decomposition of $\varphi(a)$.
\end{proposition}

\begin{remark}
\label{RationalPFremark}
The formulas \eqref{Abel32} and \eqref{Abel33} linking residues and partial fractions are well known (see, for example,  \cite[1.2.4]{CD}).  What is new here is the recognition of their relevance to Abel's formulas in the special case considered in  Theorem \ref{Abel3A}.   
\end{remark}

\begin{proof}
We first compute $\res_{x =\beta_j}\big(\frac{\varphi(x)}{a-x}\big)$.  In the partial fraction decomposition of $\varphi(x)$, isolate the $j$th summand  and expand everything else in terms of $x-\beta_j$ to obtain
\begin{equation}
\label{varphiexpbj}
\varphi(x) = \sum_{k = 1}^{\nu_j} \frac{A_j^{(k)}}{(x-\beta_j)^k} + b_0 + b_1(x-\beta_j) + \cdots.
\end{equation}
We also have the expansion
\[
\frac1{a-x} = \frac1{(a-\beta_j)\big(1 - \frac{x-\beta_j}{a-\beta_j}\big)} = \sum_{n=0}^\infty \frac{(x-\beta_j)^n}{(a-\beta_j)^{n+1}}.
\]
Some easy algebra implies that $\frac{\varphi(x)}{a-x}$ has the expansion
\[
\frac{\varphi(x)}{a-x} =  \cdots + \frac1{x-\beta_j}\Big( \frac{A^{(1)}_j}{a-\beta_j} +  \frac{A^{(2)}_j}{(a-\beta_j)^2}  + \cdots  +  \frac{A^{(\nu_j)}_j}{(a-\beta_j)^{\nu_j}}  \Big) + \cdots.
\]
Since $\res_{x=\beta_j}\big(\frac{\varphi(x)}{a-x}\big)$ is the coefficient of $\frac1{x-\beta_j}$, \eqref{Abel32} follows immediately.

For $\text{\large$\varPi$}\big(\frac{\varphi(x)}{z-x}\big)$, expand $\varphi(x)$ in terms of descending powers of $\frac1x$ to obtain
\begin{equation}
\label{varphiexpinfty}
\varphi(x) = \underbrace{c_m x^m + c_{m-1} x^{m-1} + \cdots + c_0}_{\displaystyle G(x)} + \frac{c_{-1}}{x} +  \frac{c_{-2}}{x^2} + \cdots.
\end{equation}
Note also that
\[
\frac1{a-x} = \frac{-1}x \frac1{1-\frac{a}{x}} = - \sum_{n=0}^\infty \frac{a^n}{x^{n+1}}.
\]
Then multiplying out $\frac{\varphi(x)}{a-x}$ gives the expansion
\[
\frac{\varphi(x)}{a-x} = \cdots + \frac1x \big(- c_m a^{m} - \cdots - c_1 a - c_0\big) + \cdots.
\]
Since $\text{\large$\varPi$}\big(\frac{\varphi(x)}{a-x}\big)$ is the coefficient of $\frac1x$, \eqref{Abel33} follows easily.
\end{proof}

We can now compute $\int\!\varphi(a)\,da$ using Theorem \ref{justifyint}. 

\begin{theorem}
\label{AbelIntRatThm}
\begin{equation}
\label{Abel31B}
\int \! \varphi(a)\,da = \sum_{j=1}^\alpha \res_{x=\beta_j}(\varphi(x)\log_{\beta_j}(a-x) - \text{\large$\varPi$}(\varphi(x)\log_\infty(a-x)) + C,
\end{equation}
where $\log_{\beta_j}(a-x)$ and $\log_{\infty}(a-x)$ are from Definition \ref{logxi}.  
\end{theorem}

\begin{proof}
By \eqref{Sxrat}, we obtain
\[
S(x) =  \frac{\varphi(x)  \hskip1pt da}{a-x}  =   \varphi(x) \hskip1pt \frac{ \delta(a-x)}{a-x}, 
\]
which implies 
\[
\widehat{S}_{\beta_j}(x) = \varphi(x) \log_{\beta_j}(a-x) \quad\text{and}\quad \widehat{S}_\infty(x) = \varphi(x) \log_\infty(a-x)
\]
by the formulas given in Section \ref{AvPf}.  Then \eqref{Abel31B} follows from Theorem \ref{justifyint}.
\end{proof}

\begin{remark}
\label{AbelRat}
Abel's version of  \eqref{Abel31B} is equation (201) on  \cite[p.\ 203]{A1}, which states
\begin{equation}
\label{Abel201}
\int\! \frac{fx_1\hskip.5pt.\hskip.5pt dx_1}{f_2 x_1} = C - \text{\large$\varPi$} \frac{fx}{f_2 x } \log(x-x_1) + \sum \nu \frac{d^{\nu-1}}{d\beta^{\nu-1}}\bigg(\frac{f\beta}{f_2^{(\nu)}\beta}\log(\beta-x_1)\bigg).
\end{equation}
Similar to  \eqref{Abel36} and \eqref{Abel37}, Abel's formula is not quite correct, as will be explained in Sections \ref{SecRoweAbel} and \ref{rewrite}.
After stating \eqref{Abel201}, Abel notes that his formula ``gives, as one can see, the integral of every rational differential.'' 
\end{remark}

Not surprisingly, \eqref{Abel31B} computes the integral of partial fraction decomposition of $\varphi(a)$.  
To see why, we begin with $\log_{\beta_j}(a-x)$.  Since $a-x = (a-\beta_j)\big(1 - \frac{x-\beta_j}{a-\beta_j}\big)$, Definition \ref{logxi} gives
\[
\log_{\beta_j}\hskip-1pt(a-x) = \log(a-\beta_j) - \sum_{k=1}^\infty \frac{(x-\beta_j)^k}{k(a-\beta_j)^k}.
\]
Combining this with the expansion of $\varphi(x)$ given in \eqref{varphiexpbj}, one computes without difficulty that
\[
\res_{x=\beta_j}(\varphi(x)\log_{\beta_j}\hskip-1pt(a-x)) = A^{(1)}_j \log(a-\beta_j) - \frac{A^{(2)}_j}{a-\beta_j}   - \cdots  -  \frac{A^{(\nu_j)}_j}{(\nu_j-1)(a-\beta_j)^{\nu_j-1}},
\]
which is easily seen to be an antiderivative of
\[
\sum_{k = 1}^{\nu_j} \frac{A_j^{(k)}}{(a-\beta_j)^k} = \frac{A^{(1)}_j }{a-\beta_j} + \frac{A^{(2)}_j}{(a-\beta_j)^2}   + \cdots  +  \frac{A^{(\nu_j)}_j}{(a-\beta_j)^{\nu_j}}.
\]
Similarly, $a-x = -x\big(1-\frac{a}{x}\big)$ and Definition \ref{logxi} imply that 
\[
\log_\infty(a-x) =  \log(-1) - \frac{a}{x} -  \frac12\Big(\frac{a}{x}\Big)^{\!2} - \cdots.
\]
Combining this with \eqref{varphiexpinfty} leads to
\[
\text{\large$\varPi$}(\varphi(x)\log_\infty(a-x)) = -\frac{c_m a^{m+1}}{m+1} - \cdots - \frac{c_1 a^2}{2} -c_0 a + c_{-1}\log(-1),
\]
which is clearly an antiderivative of $-G(a)$. 

\begin{remark}
We can summarize our discussion of
\[
\int \! \varphi(a)\,da = \sum_{j=1}^\alpha \res_{x=\beta_j}\big(\varphi(x)\log_{\beta_j}(a-x)\big) - \text{\large$\varPi$}\big(\varphi(x)\log_\infty(a-x)\big) + C
\]
by saying that Abel's formulas yield interesting mathematics even in the simplest cases.   When we speak of his formulas as \emph{algorithms embedded in formulas}, we see that in the special case of a rational function $\varphi(x)$, the algorithm is the standard partial fraction method for computing $\int\! \varphi(a)\,da$.
\end{remark}

\subsection{Abel's Proof for \emph{dv} and \emph{v}}
\label{SecRoweAbel}

To get a sense of what Abel did in \cite{A1}, consider Figure \ref{foureqn}, which shows equations (34), (35), (36), and (37) from \cite[pp.\ 158--159]{A1}.

\begin{figure}[H]
\[
\framebox{
\begin{picture}(325,187)
\put(0,170){(34)}
\put(30,170){$\displaystyle dv =  -\text{\large$\varPi$} \frac{R_1x}{\theta_1x . Fx} 
+ \sum{}\rule{0pt}{12pt}'\nu \frac{d^{\nu-1}}{dx^{\nu-1}} \bigg\{\frac{R_1x}{\theta_1^{(\nu)}x . Fx}\bigg\}$}
\put(70,148){$\displaystyle (x = \beta_1 \dots \beta_\alpha)$}
\put(0,124){(35)}
\put(30,124){$\displaystyle F_0(x) . (x - \beta_1)^{-k_1} (x-\beta_2)^{-k_2} \dots  (x-\beta_\alpha)^{-k_\alpha}$}
\put(78,107){$\displaystyle =(x - \beta_1)^{\mu_1-k_1} (x-\beta_2)^{\mu_2-k_2} \dots  (x-\beta_\alpha)^{\mu_\alpha-k_1\alpha} = F_2x{:}$}
\put(84,85){$\displaystyle R_1x = F_2x . Fx. \sum \frac{f_1(x,y)}{\chi'y} \frac{\delta \theta y}{\theta y}$}
\put(0,53){(36)}
\put(30,53){\small$\displaystyle dv = {-}\text{\large$\varPi$} \frac{F_2x}{\theta_1x} \sum \frac{f_1(x,y)}{\chi'y} \frac{\delta \theta y}{\theta y}
+ \sum{}\rule{0pt}{12pt}'\nu \frac{d^{\nu-1}}{dx^{\nu-1}} \bigg\{\frac{F_2x}{\theta_1^{(\nu)}x} \sum \frac{f_1(x,y)}{\chi'y} \frac{\delta \theta y}{\theta y}\bigg\}$}
\put(0,15){(37)}
\put(30,15){\small $\displaystyle v\hskip-.5pt = \hskip-.5pt C{-}\text{\large$\varPi$} \frac{F_2x}{\theta_1x} \sum \frac{f_1(x,y)}{\chi'y}  \log \theta y
{+} \sum{}\rule{0pt}{12pt}'\hskip-.5pt \nu \frac{d^{\nu-1}}{dx^{\nu-1}} \bigg\{\hskip-.5pt \frac{F_2x}{\theta_1^{(\nu)}x}\hskip-.5pt  \sum \hskip-.5pt \frac{f_1(x,y)}{\chi'y}  \log \theta y\hskip-.5pt \bigg\}$}
\end{picture}}
\]
\caption{Four Equations From Abel's Memoir}
\label{foureqn}
\end{figure}

We want to understand how Abel proved (36) and (37).   After stating (36), Abel says the following:
\begin{quote}
In this form the value $dv$ is immediately integrable, since $F_2x$, $\theta_1 x$, $f_1(x,y)$ and $\chi'y$ are all independent of of the quantities $a, a', a'' \dots$, to which the differentiation relates.  \cite[p.\ 158]{A1}
\end{quote}
This leads Abel to (37) since
\[
\delta \log \theta y = \frac{\delta \theta y}{\theta y}.
\]
Thus (37) is an immediate consequence of (36).   However, Abel also notes that substituting $R_1x$ from (35) into the expression for  $dv$ given in (34) immediately gives (36). So (36) follows from (34), i.e., (34) is the heart of the matter.

In this section, we will recast (34) in modern form as Theorem \ref{AThm} and prove the theorem using only arguments taken from \cite{A1}.  Then, in Section \ref{rewrite}, we will explain the notation in Figure \ref{foureqn} and show how 
 (34) and (36) transform into the modern versions proved in Theorems \ref{AThm} and \ref{AThmdv} respectively.

Before stating Theorem \ref{AThm}, recall from \eqref{fxyform} that $f(x,y)$ is written in the form
 \[
 f(x,y) = \frac{f_1(x,y)}{f_2(x) \chi'(y)}.
 \]
We also have the resultant $r(x)  = F_0(x) F(x)$ from \eqref{Efact}.  The roots of $f_2(x)F_0(x)$ are $\beta_1,\dots,\beta_\alpha$ with multiplicities $\nu_1,\dots,\nu_\alpha$.  We set $\theta_1(x) = f_2(x)F_0(x)$.  This agrees with a Abel's notation in a common special case.\footnote{\label{f6} More precisely, on p.\ 153 of \cite{A1}, Abel introduces exponents $k_1,\dots,k_\alpha$, and on p.\ 160, he notes that one can always assume $k_1=\cdots =k_\alpha=0$.  We will assume this, which implies that $\nu_i$ and $\theta_1(x)$ have the same meaning for both us and Abel.  See Section \ref{rewrite} for more details.} 

Our treatment of Abel's formula for $dv$ used the differential form
\[
S(x) = \sum_{\ell=1}^n f(x,y^{(\ell)}) \frac{\delta \theta(y^{(\ell)})}{\theta(y^{(\ell)})}.
\]
Using Lemma \ref{AbelLem} and $\theta_1(x) = f_2(x)F_0(x)$, we have
\begin{equation}
\label{Sxtheta1}
S(x) = \frac{R(x)}{f_2(x)F_0(x)F(x)} =  \frac{R(x)}{\theta_1(x)F(x)},
\end{equation}
which easily implies that 
\begin{equation}
\label{n81}
R(x) = F_0(x) F(x) \sum_{\ell=1}^n \frac{f_1(x,y^{(\ell)})}{\chi'(y^{(\ell)})} \frac{\delta \theta(y^{(\ell)})}{\theta(y^{(\ell)})}. 
\end{equation}
For Abel, this is \cite[(19) on p.\ 153]{A1} since $r(x) = F_0(x)F(x)$.  The proof of Lemma~\ref{AbelLem} given in Section~\ref{APf} is based on his arguments.

We are now ready to state and prove the modern version of (34).

\begin{theorem}
\label{AThm}
Assume \eqref{Assume1} and let $\beta_j$ and $\nu_j$ be as above.  Then
\[
\sum_{i=1}^\mu f(x_i,y_i)\, dx_i = \sum_{j=1}^\alpha \frac1{(\nu_j-1)!}  \bigg( \frac{d^{\nu_j-1}}{dx^{\nu_j-1}}\frac{(x- \beta_j)^{\nu_j} R(x)}{\theta_1(x)F(x)} \bigg)\Big|_{x=\beta_j} \!\! - \text{\large$\varPi$} \frac{R(x)}{\theta_1(x) F(x)}. 
\]
\end{theorem}

\begin{remark}
\label{Thm81rmk}
Section \ref{rewrite} will explain how to transform Abel's formula (34) in Figure \ref{foureqn} into 
Theorem \ref{AThm}.
\end{remark}

Abel makes extensive use of partial fraction decompositions; residues play a very minor role.  
At the heart of Abel's proof are two preliminary results that are unnumbered equations in \cite{A1}.  To highlight their importance, we state them as lemmas, whose proofs will be deferred to the end of the section.  

\begin{lemma}
\label{Abel1}
Recall that $F(x)$ has roots $x_1,\dots,x_\mu$, all of multiplicity one by \eqref{Assume1}.  Then, for any polynomial $F_1(x)$,
\[
\sum_{i=1}^\mu \frac{F_1(x_i)}{F'(x_i)} = \text{\large$\varPi$}\Big(\frac{F_1(x)}{F(x)}\Big).
\]
\end{lemma}

\begin{lemma}
\label{Abel2}
Suppose $\theta_1(x)$ has roots $\beta_1,\dots,\beta_\alpha$ with multiplicities $\nu_1,\dots,\nu_\alpha$.  If $R_3(x)$ is a polynomial with $\deg(R_3(x)) < \deg(\theta_1(x))$ and $\beta$ is a new variable, then
\[
\frac{R_3(x)}{\theta_1(x)} = \sum_{j=1}^\alpha \frac1{(\nu_j-1)!} \bigg(\frac{d^{\nu_j-1}}{d\beta^{\nu_j-1}} \frac{(\beta- \beta_j)^{\nu_j} R_3(\beta)}{(x-\beta)\theta_1(\beta)} \bigg)\Big|_{\beta=\beta_j}.
\]
\end{lemma}

Abel's versions of these lemmas appear on pages 155 and 157 of \cite{A1}.  His statement of Lemma \ref{Abel2} is marred by the same error noted in \eqref{Abel36}.   

Assuming Lemmas \ref{Abel1} and \ref{Abel2}  for now, here is Abel's proof of Theorem \ref{AThm}.
 
\begin{proof}[Proof of Theorem \ref{AThm}]
We begin with
\[
f(x_i,y_i)\,dx_i = -\res_{x_i}(S(x)) = -\res_{x_i}\Big( \frac{R(x)}{\theta_1(x) F(x) }\Big) = - \frac{R(x_i)}{\theta_1(x_i) F'(x_i) }.
\]
The first equality is Lemma \ref{fxiyiSxi}, the second uses \eqref{Sxtheta1}, and the third follows from \eqref{computeres1} since $\theta_1(x_i) = f_2(x_i)F_0(x_i) \ne 0$.  Thus
\begin{equation}
\label{n86}
\sum_{i=1}^\mu f(x_i,y_i)\,dx_i = -\sum_{i=1}^\mu \frac{R(x_i)}{\theta_1(x_i) F'(x_i) }.
\end{equation}
This is equation (22) of \cite[p.\ 153]{A1}.  If you take the proof of \eqref{n86} just given and strip away everything to do with residues, what remains is entirely due to Abel.

Now divide $R(x)$ by $\theta_1(x)$ to get polynomial differential forms $R_2(x)$ and $R_3(x)$ such that
\[
\frac{R(x)}{\theta_1(x)} = R_2(x) + \frac{R_3(x)}{\theta_1(x)}, \ \ \deg(R_3(x)) < \deg(\theta_1(x)). 
\]
Then \eqref{n86} becomes
\begin{equation}
\label{n87}
\sum_{i=1}^\mu f(x_i,y_i)\,dx_i = -\sum_{i=1}^\mu \frac{R_2(x_i)}{F'(x_i) }  -\sum_{i=1}^\mu \frac{R_3(x_i)}{\theta_1(x_i) F'(x_i) }.
\end{equation}
By Lemma \ref{Abel1}, the first sum on the right-hand side is
\begin{equation}
\label{n88}
\sum_{i=1}^\mu \frac{R_2(x_i)}{F'(x_i) } = \text{\large$\varPi$}\Big(\frac{R_2(x)}{F(x) }\Big) =  \text{\large$\varPi$}\Big(\frac{R_2(x)}{F(x)} + \frac{R_3(x)}{\theta_1(x) F(x) }\Big)= \text{\large$\varPi$}\Big(\frac{R(x)}{\theta_1(x) F(x)}\Big).
\end{equation}
The second equality follows from $\deg(R_3(x)) < \deg(\theta_1(x))$ and $0 < \deg(F(x))$.  

For the second sum on the right of \eqref{n87}, observe that by Lemma \ref{Abel2},
\begin{align*}
\frac{R_3(x_i)}{\theta_1(x_i) F'(x_i)}  &= \frac1{F'(x_i)}  \frac{R_3(x_i)}{\theta_1(x_i)}\\ &=  \frac1{F'(x_i)}  
\sum_{j=1}^\alpha \frac1{(\nu_j-1)!} \bigg(\frac{d^{\nu_j-1}}{d\beta^{\nu_j-1}} \frac{(\beta- \beta_j)^{\nu_j} R_3(\beta)}{(x_i-\beta)\theta_1(\beta)} \bigg)\Big|_{\beta=\beta_j} \\
&= \sum_{j=1}^\alpha \frac1{(\nu_j-1)!} \bigg(\frac{d^{\nu_j-1}}{d\beta^{\nu_j-1}} \frac{(\beta- \beta_j)^{\nu_j} R_3(\beta)}{\theta_1(\beta)} \cdot \frac1{ (x_i-\beta) F'(x_i)}    \bigg)\Big|_{\beta=\beta_j}.
\end{align*}
When we add this up for $i = 1,\dots,\mu$, the quantity inside the large parentheses is
\[
\frac{d^{\nu_j-1}}{d\beta^{\nu_j-1}} \frac{(\beta- \beta_j)^{\nu_j} R_3(\beta)}{\theta_1(\beta)} \sum_{i=1}^\mu \frac1{ (x_i-\beta) F'(x_i)}.
\]
However, we have a classic partial fraction decomposition known to Abel:
\begin{equation}
\label{n89}
\frac1{F(\beta)} = \sum_{i=1}^\mu \frac1{ (\beta - x_i) F'(x_i)} = - \sum_{i=1}^\mu \frac1{ (x_i-\beta) F'(x_i)}.
\end{equation}
This holds because $F(x)$ has simple roots by assumption \eqref{Assume1}.  It follows that
\begin{align*}
\sum_{i=1}^\mu \frac{R_3(x_i)}{\theta_1(x_i) F'(x_i) } &= -\sum_{j=1}^\alpha \frac1{(\nu_j-1)!} \bigg(\frac{d^{\nu_j-1}}{d\beta^{\nu_j-1}} \frac{(\beta- \beta_j)^{\nu_j} R_3(\beta)}{\theta_1(\beta) F(\beta)} \bigg)\Big|_{\beta=\beta_j}\\
&= -\sum_{j=1}^\alpha \frac1{(\nu_j-1)!} \bigg(\frac{d^{\nu_j-1}}{dx^{\nu_j-1}} \frac{(x- \beta_j)^{\nu_j} R_3(x)}{\theta_1(x) F(x)} \bigg)\Big|_{x=\beta_j}.
\end{align*}
The second line follows from the first by replacing the variable $\beta$ with the variable $x$.  However, in this formula, we can replace $R_3(x)$ with $R(x) = R_3(x) + \theta_1(x) R_2(x) $ because
\[
\bigg(\frac{d^{\nu_j-1}}{dx^{\nu_j-1}} \frac{(x- \beta_j)^{\nu_j}  \theta_1(x) R_2(x)}{\theta_1(x) F(x)} \bigg)\Big|_{x=\beta_j}
\! = \bigg(\frac{d^{\nu_j-1}}{dx^{\nu_j-1}} (x- \beta_j)^{\nu_j} \frac{R_2(x)}{F(x)} \bigg)\Big|_{x=\beta_j} \! = 0
\]
since $\beta_j$ is not a root of $F(x)$.  Thus
\begin{equation}
\label{n88a}
\sum_{i=1}^\mu \frac{R_3(x_i)}{\theta_1(x_i) F'(x_i) } = -\sum_{j=1}^\alpha \frac1{(\nu_j-1)!} \bigg(\frac{d^{\nu_j-1}}{dx^{\nu_j-1}} \frac{(x- \beta_j)^{\nu_j} R(x)}{\theta_1(x) F(x)} \bigg)\Big|_{x=\beta_j}.
\end{equation}

We are now done, since  
\begin{align*}
\sum_{i=1}^\mu f(x_i,y_i)\,dx_i &= -\sum_{i=1}^\mu \frac{R_2(x_i)}{F'(x_i) }  -\sum_{i=1}^\mu \frac{R_3(x_i)}{\theta_1(x_i) F'(x_i) }\\
&= - \text{\large$\varPi$}\Big(\frac{R(x)}{\theta_1(x) F(x)}\Big) + \sum_{j=1}^\alpha \frac1{(\nu_j{-}1)!}  \bigg( \frac{d^{\nu_j-1}}{dx^{\nu_j-1}}\frac{(x{-} \beta_j)^{\nu_j} R(x)}{\theta_1(x)F(x)} \bigg)\Big|_{x=\beta_j},
\end{align*}
where the first line is \eqref{n87} and the second line follows from  \eqref{n88} and \eqref{n88a}.
\end{proof}

It remains to give Abel's proofs of Lemmas \ref{Abel1} and \ref{Abel2}.   As promised, partial fractions feature prominently in both proofs.  

\begin{proof}[Proof of Lemmma \ref{Abel1}] For us, this is an easy consequence of \eqref{computeres1} and the Global Residue Theorem (apply \eqref{GRT} to $\frac{F_1(x)}{F(x)}\hskip1pt dx$).  Abel's proof is quite different.  He begins with the partial fraction decompostion \eqref{n89} where $\beta$ is replaced by $\alpha$ and then expands $\frac1{\alpha - x_i}$ as a power series in $\frac{1}{\alpha}$.  This gives
\[
\frac1{F(\alpha)} = \sum_{i=1}^\mu \frac1{ (\alpha - x_i) F'(x_i)} = \sum_{i=1}^\mu \!\bigg(\!\sum_{m=0}^\infty \frac{x_i^m}{\alpha^{m+1}}\!\bigg)\! \frac1{F'(x_i)} = \sum_{m=0}^\infty \!\bigg(\!\sum_{i=1}^\mu \frac{x_i^m}{F'(x_i)}\!\bigg)\! \frac1{\alpha^{m+1}}.
\]
Here is what Abel says next \cite[p.\ 155]{A1}:
\begin{quote}
\dots\ it follows that $\sum \frac{x^m}{F'(x)}$ [$= \sum_{i=1}^\mu \frac{x_i^m}{F'(x_i)}$] is equal to the coefficient of $\frac{1}{\alpha^{m+1}}$ in the development of the function $\frac1{F(\alpha)}$, or, what amounts to the same, to that of $\frac1\alpha$ in the development of $\frac{\alpha^m}{F(\alpha)}$.  By designating therefore by $\text{\large$\varPi$}F_1x$ the coefficient of $\frac1x$ in the development of any function $F_1x$ according to the descending powers of $x$, one will have
\[
\sum \frac{x^m}{F'(x)} = \text{\large$\varPi$}\frac{x^m}{F(x)}.
\]
\end{quote}
The last sentence is where Abel defines {\large$\varPi$}.  The $\frac1\alpha$ and $\frac1x$  in this quote are the only places where a residue (here a residue at $\infty$) appears explicitly in his Paris memoir.

The upshot is that Abel has proved that
\[
\sum_{i=1}^\mu \frac{x_i^m}{F'(x_i)} = \text{\large$\varPi$}\Big(\frac{x^m}{F(x)}\Big)
\]
for any integer $m \ge 0$.  From here, Lemma \ref{Abel1} follows by linearity.
\end{proof}

We now turn to Lemma \ref{Abel2}, which again uses partial fractions.

\begin{proof}[Proof of Lemma \ref{Abel2}] 
Since $\deg(R_3(x)) < \deg(\theta_1(x))$, we have a partial fraction expansion
\[
\frac{R_3(x)}{\theta_1(x)} = \sum_{j=1}^\alpha \sum_{k = 1}^{\nu_j} \frac{A_j^{(k)}}{(x-\beta_j)^k},
\]
where it is well known that
\[
A_j^{(k)} = \frac{1}{(\nu_j - k)!} \bigg(\frac{d^{\nu_j-k}p(x)}{dx^{\nu_j-k}}\bigg)\Big|_{x = \beta_j}, \ p(x) = (x - \beta_j)^{\nu_j} \frac{R_3(x)}{\theta_1(x)}
\]
(see, for example, \cite[(7-84) on p.\ 175]{V}).  Following Abel \cite[p.\ 156]{A1}, we switch from $x$ to $\beta$ and write
\begin{equation}
\label{Aerror}
A_j^{(k)} = \frac{1}{(\nu_j - k)!} \bigg(\frac{d^{\nu_j-k}p(\beta)}{d\beta^{\nu_j-k}}\bigg)\Big|_{\beta = \beta_j}, \ p(\beta) = (\beta - \beta_j)^{\nu_j} \frac{R_3(\beta)}{\theta_1(\beta)}.
\end{equation}
Abel then sets
\[
q(\beta) = \frac1{x-\beta}
\]
and observes that
\[
\frac1{(x-\beta_j)^k} = \frac1{(k-1)!} \bigg(\frac{d^{k-1}q(\beta)}{d\beta^{k-1}}\bigg)\Big|_{\beta = \beta_j}.
\]
Hence
\begin{align*}
\sum_{k = 1}^{\nu_j} \frac{A_j^{(k)}}{(x-\beta_j)^k} &= \bigg(\sum_{k = 1}^{\nu_j} \frac1{(\nu_j - k)!} \frac{d^{\nu_j-k}p(\beta)}{d\beta^{\nu_j-k}} \cdot \frac1{(k-1)!} \frac{d^{k-1}q(\beta)}{d\beta^{k-1}}\bigg)\Big|_{\beta = \beta_j}\\
&= \frac{1}{(\nu_j-1)!} \bigg(\sum_{k = 1}^{\nu_j} \binom{\nu_j-1}{k-1} \frac{d^{\nu_j-k}p(\beta)}{d\beta^{\nu_j-k}} \cdot  \frac{d^{k-1}q(\beta)}{d\beta^{k-1}}\bigg)\Big|_{\beta = \beta_j}.
\end{align*}
By the generalized product rule, the quantity inside the large parentheses is
\[
\frac{d^{\nu_j-1}}{d\beta^{\nu_j-1}} p(\beta)q(\beta) = \frac{d^{\nu_j-1}}{d\beta^{\nu_j-1}} \frac{(\beta- \beta_j)^{\nu_j} R_3(\beta)}{(x-\beta)\theta_1(\beta)}.
\]
Putting everything together, Abel obtains the desired formula
\[
\frac{R_3(x)}{\theta_1(x)} = \sum_{j=1}^\alpha \sum_{k = 1}^{\nu_j} \frac{A_j^{(k)}}{(x-\beta_j)^k} = \sum_{j=1}^\alpha  \frac{1}{(\nu_j-1)!} \bigg(\frac{d^{\nu_j-1}}{d\beta^{\nu_j-1}} \frac{(\beta- \beta_j)^{\nu_j} R_3(\beta)}{(x-\beta)\theta_1(\beta)} \bigg)\Big|_{\beta = \beta_j}.\qedhere
\]
\end{proof}

\begin{remark}
\label{AbelErrorRmk}
In the course of this proof, Abel introduces the error that reappears multiple times in \cite{A1}. Specifically, he gives the correct formula for $p = p(\beta)$ in \eqref{Aerror} but follows with the formula
\begin{equation}
\label{wrongp}
p = \frac{\Gamma(\nu+1) R_3\beta}{\theta_1^{(\nu)}\beta},
\end{equation}
where $\Gamma(\nu+1)$ is ``the product $1.2.3 \dots (\nu-1).\nu$'' (this is how Abel writes factorials in \cite{A1}), and $\theta_1^{(\nu)}(x)$ is ``the $\nu$\textsuperscript{th} derivative of the function $\theta_1 x$ with respect to $x$" (this is incorrect)  \cite[p.\ 156]{A1} .  We will say more about Abel's error in Section \ref{rewrite}.
\end{remark}

Lemma \ref{Abel2} is quite remarkable, especially when we recognize that 
\[
\frac{1}{(\nu_j-1)!} \bigg(\frac{d^{\nu_j-1}}{d\beta^{\nu_j-1}} \frac{(\beta- \beta_j)^{\nu_j} R_3(\beta)}{(x-\beta)\theta_1(\beta)} \bigg)\Big|_{\beta = \beta_j} \!\!= \res_{\beta =\beta_j}\Big(\frac{R_3(\beta)}{(x-\beta)\theta_1(\beta)}\Big)
\]
by \eqref{computeres}.  Hence the lemma shows that
\begin{align*}
\sum_{k = 1}^{\nu_j} \frac{A_j^{(k)}}{(x{-}\beta_j)^k} &= \frac{1}{(\nu_j{-}1)!} \bigg(\frac{d^{\nu_j-1}}{d\beta^{\nu_j-1}} \frac{(\beta{-} \beta_j)^{\nu_j} R_3(\beta)}{(x{-}\beta)\theta_1(\beta)} \bigg)\Big|_{\beta = \beta_j}\!\!
= \res_{\beta =\beta_j}\Big(\frac{R_3(\beta)}{(x{-}\beta)\theta_1(\beta)}\Big),
\end{align*}
which is a version of \eqref{Abel32} from Proposition \ref{AbelPF1}. In Section \ref{SecRational}, partial fractions arose naturally in a special case of Abel's formula.  Here, Abel's proof of Lemma~\ref{Abel2} shows that the connection with partial fractions goes much deeper.  They are an essential tool in his proof.

\begin{remark}
\label{nui}
The reader may note that in the above formulas, $\nu_j$ is the order of $\beta_j$ as a root of $\theta_1(\beta)$, which may be strictly greater than the order of $\beta_j$ as pole of $\frac{R_3(\beta)}{(x-\beta)\theta_1(\beta)}$ since we do not assume that $\theta_1(\beta)$ is relatively prime to $R_3(\beta)$.  Fortunately, this is harmless.  When $\nu_j$ is strictly greater than the order of $\beta_j$ as a pole, \eqref{computeres} and \eqref{Aerror} continue to be valid, so  everything still works. 
\end{remark}

This completes our discussion of Abel's proof.  His methods are elementary and use partial fractions instead of residues.  Nevertheless, his proof displays the wonderful ingenuity that is so typical of Abel.

\subsection{Abel's Original Formula for \emph{dv}}
\label{rewrite}
Recall Abel's equations (34), (35), and (36) from Figure \ref{foureqn}.
The first task of this section is to explain how Abel's formula  (36) becomes the modern version
\begin{equation}
\label{NA1}
dv = {-} \text{\large$\varPi$}\Big( \sum_{\ell=1}^n f(x,y^{(\ell)}) \hskip1pt\frac{\delta \theta(y^{(\ell)})}{\theta(y^{(\ell)})} \Big) +\sum_{j=1}^\alpha \res_{\beta_j}\hskip-1pt \Big(\sum_{\ell=1}^n f(x,y^{(\ell)}) \hskip1pt\frac{\delta \theta(y^{(\ell)})}{\theta(y^{(\ell)})}\Big) 
\end{equation}
from \eqref{ATModern} and \eqref{ATModernS}, which we proved in Theorem \ref{AThmdv}.  We also need to understand the error in (36) and how the modern version \eqref{NA1} corrects the error.  Our second task is to understand how (34) becomes the modern version proved in Theorem~\ref{AThm}.

In the previous section, we showed that (34) leads to (36) via (35).  The challenge is that  (34), (35), and (36) contain a \emph{lot} of notation that needs to be explained.  We have some work to do.

We already know $F(x)$, $f_1(x,y)$, $\chi'(y)$, $\theta(y)$, and $\delta\theta(y)$.  The next step is to untangle $R_1(x)$, $F_2(x)$, $\theta_1(x)$, $\mu_i$, and $k_i$. (We ignore $\theta_1^{(\nu)}(x)$ for now since this is where the error occurs.)  We turn to equation (23) on \cite[p.\ 153]{A1}, which says
\begin{equation}
\label{Abel23}
\begin{aligned}
F_0(x) &= (x-\beta_1)^{\mu_1}\cdots (x-\beta_\alpha)^{\mu_\alpha}\\
f_2(x) &= (x-\beta_1)^{m_1}\cdots (x-\beta_\alpha)^{m_\alpha} A\\
R(x) &= (x-\beta_1)^{k_1}\cdots (x-\beta_\alpha)^{k_\alpha} R_1(x).
\end{aligned} 
\end{equation}
Here, $F_0(x), f_2(x), R(x)$ have the same meaning as for us.  Also, $A$ is a constant and $\mu_i, m_i,k_i$ are nonnegative integers.  Then, on p.\ 154, Abel defines $\nu_i = \mu_i+ m_i -k_i$ and sets
\[
\theta_1(x) =  A(x-\beta_1)^{\nu_1}\cdots (x-\beta_\alpha)^{\nu_\alpha}.
\]
In \cite[Vol.\ 2, pp.\ 294--295]{A5}, Lie and Sylow make some useful remarks about \eqref{Abel23}. Most relevant for us is their pointer to \cite[p.\ 160]{A1}, where Abel discusses a situation where $k_1 = \cdots = k_\alpha =0$ and adds the parenthetical comment  ``(one can make the same supposition in all cases).''  This is true because the $\beta_j$ in \eqref{Abel23} are the roots of $f_2(x)F_0(x)$, yet the key player is $S(x) = \frac{R(x)}{f_2(x)F_0(x) F'(x)}$ from Lemma \ref{AbelLem}.  Abel uses the $k_i$ to cancel common factors of $R(x)$ and $f_2(x)F_0(x)$.  But cancellation is not necessary, which allows us to set $k_1 = \cdots = k_\alpha =0$.  As indicated in footnote~\ref{f6}, we will do so.  This has several nice consequences, including 
\begin{equation}
\label{NiceCons}
\begin{aligned}
R_1(x) &= R(x)\\
\theta_1(x) &= f_2(x) F_0(x) = (x-\beta_1)^{\nu_1}\cdots (x-\beta_\alpha)^{\nu_\alpha}.
\end{aligned}
\end{equation}
This explains why we could set $\theta_1(x) = f_2(x) F_0(x)$ in Section \ref{SecRoweAbel}.  Looking back at (35) in Figure~\ref{foureqn}, we see that $k_1 = \cdots = k_\alpha =0$ also implies that $F_2(x) = F_0(x)$.

Before we can transform Abel's original formula
\[
dv =  {-}\text{\large$\varPi$} \frac{F_2x}{\theta_1x} \sum \frac{f_1(x,y)}{\chi'y} \frac{\delta \theta y}{\theta y}
+ \sum{}\rule{0pt}{12pt}'\nu \frac{d^{\nu-1}}{dx^{\nu-1}} \bigg\{\frac{F_2x}{\theta_1^{(\nu)}x} \sum \frac{f_1(x,y)}{\chi'y} \frac{\delta \theta y}{\theta y}\bigg\},
\]
there are three further things to explain:

\smallskip

\noindent {\scshape First}:\ The summations in the formula can be explained as follows:
\begin{itemize}
\item The ordinary sum $\sum$ appears twice, and as we have seen, it is the sum for $\ell = 1,\dots,n$ over the roots $y^{(\ell)}$ of $\chi(x,y)$.
\item The sum $\sum{}\rule{0pt}{8pt}'$ is explained in (34) in Figure \ref{foureqn}:\ it is the sum for $j = 1,\dots,\alpha$ over the roots $\beta_j$ of $\theta_1(x) = f_2(x)F_0(x)$, and $\nu$ is really $\nu_j$, which by \eqref{NiceCons} is the multiplicity of $\beta_j$ as a root of $f_2(x)F_0(x)$.  
\end{itemize}

\smallskip

\noindent {\scshape Second}:\  The large $\big\{$ and $\big\}$ in the  formula are easy to explain: they indicate that the derivative is evaluated at $x = \beta_j$.

\smallskip

\noindent {\scshape Third}:\ Finally, we explain $\theta_1^{(\nu)}x$.  As noted in Remark \ref{AbelErrorRmk}, Abel's description of $\theta_1^{(\nu)}(x )$ as the $\nu$\textsuperscript{th} derivative of $\theta_1(x)$ is incorrect.  The correct description is easy to find, for as we saw in \eqref{Aerror} and \eqref{wrongp},
\[
p(\beta) =(\beta-\beta_j)^{\nu_j}  \frac{R_3(\beta)}{\theta_1(\beta)} \ \text{ and } \ p(\beta) = \frac{\nu_j! \hskip1pt R_3(\beta)}{\theta_1^{(\nu_j)}(\beta)}.
\]
Solving for $\theta_1^{(\nu_j)}(\beta)$ gives the correction suggested in \cite[Vol.\ 2, p.\ 295]{A5}, namely
\[
\theta_1^{(\nu_j)}(\beta) = \nu_j! \frac{\theta_1(\beta)}{(\beta- \beta_j)^{\nu_j}}, \ \text{i.e.,} \ \, \theta_1^{(\nu_j)}(x) = \nu_j! \frac{\theta_1(x)}{(x- \beta_j)^{\nu_j}} = \nu_j! \frac{f_2(x)F_0(x)}{(x- \beta_j)^{\nu_j}} 
\]

We are have everything we need to transform Abel's original formula for $dv$ into the modern version proved in Theorem \ref{AThmdv}:
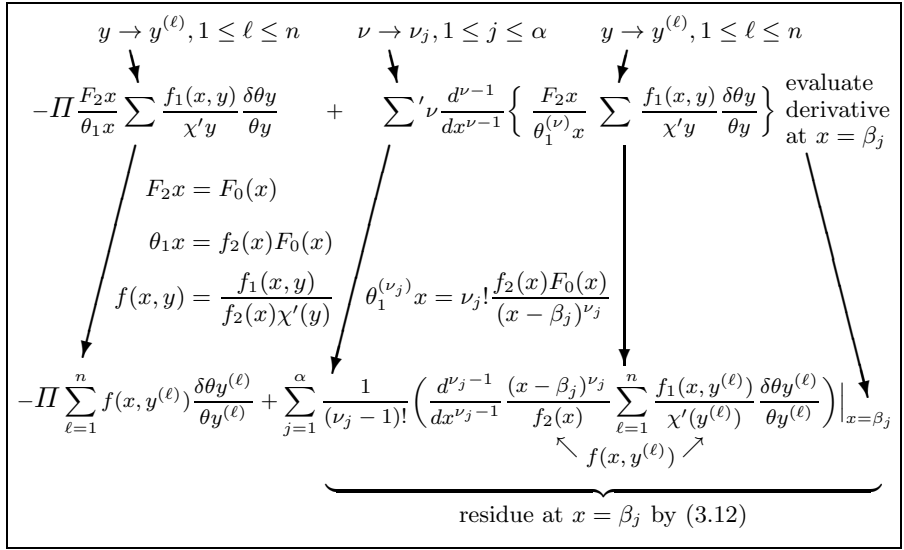
\begin{figure}[H]
\[
\framebox{
\begin{picture}(325,200)
\put(0,50){\footnotesize $\displaystyle \!\! {-}\text{\large$\varPi$}\sum_{\ell=1}^n f(x,y^{(\ell)}) \frac{\delta \theta y^{(\ell)}}{\theta y^{(\ell)}} + \sum_{j=1}^\alpha \frac1{(\nu_j-1)!} \bigg(\frac{d^{\nu_j-1}}{dx^{\nu_j-1}}\frac{(x-\beta_j)^{\nu_j}}{f_2(x)} \sum_{\ell=1}^n \frac{f_1(x,y^{(\ell)})}{\chi'(y^{(\ell)})} \frac{\delta \theta y^{(\ell)}}{\theta y^{(\ell)}}\bigg)\Big|_{x = \beta_j}$}
\put(0,160){\footnotesize \ \ $\displaystyle \!\!  {-}\text{\large$\varPi$} \frac{F_2x}{\theta_1x}  \sum \frac{f_1(x,y)}{\chi'y} \frac{\delta \theta y}{\theta y} \!\!\!\qquad\ +\qquad \!\!\!\sum\strut' \nu \frac{d^{\nu-1}}{dx^{\nu-1}} \bigg\{\ \frac{F_2x}{\theta_1^{(\nu)}x} \  \sum  \,\frac{f_1(x,y)}{\chi'y} \frac{\delta \theta y}{\theta y}\bigg\} $}
\put(28,190){\small $y \to y^{(\ell)}, 1 \le \ell \le n$}
\put(218,190){\small $y \to y^{(\ell)}, 1 \le \ell \le n$}
\put(126,190){\small $\nu \to \nu_j, 1 \le j \le \alpha$}
\put(45.7,130){\small $F_2x = F_0(x)$}
\put(47.2,110){\small $\theta_1x = f_2(x)F_0(x)$}
\put(34,89){\small $\displaystyle f(x,y) = \frac{f_1(x,y)}{f_2(x)\chi'(y)}$}
\put(129,89){\small $\displaystyle \theta_1^{(\nu_j)}x = \nu_j! \frac{f_2(x)F_0(x)}{(x-\beta_j)^{\nu_j}}$}
\put(288,170){\small evaluate}
\put(288,160){\small derivative}
\put(288,150){\small at $x = \beta_j$}
\put(201,29){\footnotesize $\text{\raisebox{4pt}{$\nwarrow$}} \ f(x,y^{(\ell)}) \  \hskip-.5pt \text{\raisebox{4pt}{$\nearrow$}}$}
\put(114,22){\small $\underbrace{\quad\hskip2.65in\quad}_{\text{\small residue at $x = \beta_j$ by \eqref{computeres}}}$}
\thicklines
\put(42,149){\vector(-1,-4){20}}
\put(227,149){\vector(0,-1){83}}
\put(296,146){\vector(1,-4){23}}
\put(138,149){\vector(-1,-4){21.5}}
\put(40,185){\vector(1,-4){3}}
\put(138,185){\vector(1,-4){3}}
\put(228,185){\vector(-1,-4){3}}
\end{picture} }
\]
\caption{Transforming Abel's Formula for $dv$}
\label{AbeltoModern}
\end{figure}

In Figure \ref{AbeltoModern}, Abel's original formula (36) for $dv$ at the top.  Furthermore, when the expression at the bottom of the figure is expressed using residues, it becomes
\[
dv =  {-}\text{\large$\varPi$}\sum_{\ell=1}^n f(x,y^{(\ell)}) \frac{\delta \theta y^{(\ell)}}{\theta y^{(\ell)}} + 
\sum_{j=1}^\alpha \res_{\beta_j}\Big(\sum_{\ell=1}^n f(x,y^{(\ell)}) \frac{\delta \theta y^{(\ell)}}{\theta y^{(\ell)}} \Big),
\]
which is exactly the modern version \eqref{NA1} that we proved in Theorem~\ref{AThmdv}. We now understand Abel's original formula (36) for $dv$.  In a similar way, it is straightforward to transform (34) into the formula proved in Theorem \ref{AThm}.  We omit the details.

\subsection{Boole and Rowe} 
\label{BooleRowe}
Abel's Paris memoir contains a large number of powerful ideas that were studied after its publication in 1841.  In particular, the papers by Boole \cite{B} in 1857 and Rowe \cite{R} in 1881 considered Abel's formulas for $dv$ and $v$.  Their main contribution was to simplify his proofs by judicious use of residues.

The key feature of Boole's 1857 paper \cite{B} is the symbolic operator $\Theta$, which he introduces as follows:
\begin{quote}
As respects the methods and processes which will be employed in this paper, the only peculiarity to which it seems  necessary to direct attention, is the introduction of a symbol [the $\Theta$ just mentioned], differing in interpretation only by the addition of one element, from that which Cauchy has employed in his `Calculus of Residues'. \cite[pp.\ 745--746]{B}
\end{quote}
Boole's use of a symbol like $\Theta$ is not surprising since he was one of the creators of symbolic logic.  He was a firm believer in the power of symbolic operators.

The definition of $\Theta$ involves residues, including residues at $\infty$.  Section \ref{SecRoweAbel} gave Abel's definition of {\large$\varPi$} as the coefficient of $\frac1x$ is an expansion of descending powers of $x$.  In \cite{B}, Boole  uses $C_{\frac1x}$ in place of {\large$\varPi$}. This is  the ``addition of one element''  in the above quote.  Here is a modern version of Boole's definition of $\Theta$ \cite[p.\ 752]{B}.

\begin{definition}
\label{Thetadef}
Let $\varphi(x)$ and $f(x)$ be rational functions in $\LL(x)$, where $\LL$ is algebraically closed.  Then
\[
\Theta[\varphi(x)]f(x) \ = \!\!\!\sum_{\beta\in \LL\text{ pole of }\varphi(x)}\!\!\! \res_\beta(\varphi(x) f(x)) \ - \ C_{\frac1x} (\varphi(x)f(x)).
\]
\end{definition}

Boole then proved some basic properties of $\Theta$.  In particular, on \cite[p.\ 755]{B}, he showed that if 
 $\varphi(x)$ a rational function and $f(x)$ is a polynomial, then 
\begin{equation}
\label{BooleGRT}
\Theta[\varphi(x)]f(x) = 0.
\end{equation}
When applied to the constant polynomial $f(x)=1$, this implies the Global Residue Theorem for $\varphi(x)$.  Not surprisingly, Boole's proof used partial fractions, similar to the Serre's proof of \eqref{GRT} in \cite[Lemma 3 on p.\ 21]{Serre}.

Boole considered some special cases of Abel's formulas, but the general case had to wait for the 1881 paper \cite{R} of Rowe.  The symbol $\Theta$ enabled Rowe to restate Abel's formula \eqref{ATModern} more concisely.  Using $f(x,y) = \frac{f_1(x,y)}{f_2(x)\chi'(y)}$, note that
\begin{align*}
\frac{1}{f_2(x)F_0(x)} \cdot F_0(x) \sum_{\ell=1}^n \frac{f_1(x,y^{(\ell)})}{\chi'(y^{(\ell)})} \hskip1pt \frac{\delta \theta(y^{(\ell)})}{\theta(y^{(\ell)})} &=  \sum_{\ell=1}^n \frac{f_1(x,y^{(\ell)})}{f_2(x)\chi'(y^{(\ell)})} \hskip1pt \frac{\delta \theta(y^{(\ell)})}{\theta(y^{(\ell)})} \\
&= \sum_{\ell=1}^n f(x,y^{(\ell)}) \hskip1pt \frac{\delta \theta(y^{(\ell)})}{\theta(y^{(\ell)})}.
\end{align*}
Now recall the version of Abel's formula for $dv$ stated in \eqref{ATModern}:
\[
dv =  \sum_{j=1}^\alpha \res_{\beta_j}\hskip-1pt \Big(\sum_{\ell=1}^n f(x,y^{(\ell)}) \hskip1pt\frac{\delta \theta(y^{(\ell)})}{\theta(y^{(\ell)})}\Big) -  \text{\large$\varPi$}\Big( \sum_{\ell=1}^n f(x,y^{(\ell)}) \hskip1pt\frac{\delta \theta(y^{(\ell)})}{\theta(y^{(\ell)})} \Big),
\]
where the $\beta_j$ are the roots of $f_2(x) F_0(x)$, i.e,, the poles of $\frac{1}{f_2(x)F_0(x)}$.  By the definition of $\Theta$, it follows that \eqref{ATModern} is equivalent to the equation
\begin{equation}
\label{AT1}
dv =  \Theta\Big[\frac{1}{f_2(x)F_0(x)}\Big] \Big(\!F_0(x) \sum_{\ell=1}^n \frac{f_1(x,y^{(\ell)})}{\chi'(y^{(\ell)})} \hskip1pt \frac{\delta \theta(y^{(\ell)})}{\theta(y^{(\ell)})}\Big).
\end{equation}
Rowe assumes \eqref{Assume1} and gives a complete proof of \eqref{AT1} on p.\ 720 of \cite{B}.  His proof does not use the Global Residue Theorem, but it is definitely  in the background.  As noted in  Section \ref{Limits}, Forsyth \cite[pp.\ 579--586]{F} discusses Abel's formula for $v$.  His argument follows Rowe (including the use of $\Theta$) but takes place in a purely analytic context.

The papers of Boole and Rowe were the inspiration for these notes.  

\bibliographystyle{amsplain}

\end{document}